\documentclass[a4paper,10pt,twoside,reqno]{amsart}

\usepackage[left=2cm,right=2cm,top=2cm,bottom=2cm]{geometry}
\usepackage{verbatim}
\usepackage[utf8]{inputenc}
\usepackage[OT1]{fontenc}
\usepackage[english]{babel}
\usepackage{amsfonts}
\usepackage{amsmath}
\usepackage{amssymb}
\usepackage{amsthm}
\usepackage{amsaddr}
\usepackage{enumerate}
\usepackage{color}
\usepackage{esint}
\usepackage{hyperref}
\usepackage{cite}
\usepackage{mathrsfs}
\usepackage{cleveref}
\usepackage{bbm}
\usepackage{graphicx}
\usepackage{cancel}
\usepackage{tikz-cd}

\usepackage{todonotes}

\newcommand{\ee}{\varepsilon}

\newcommand{\dd}{d}

\newcommand{\ZZ}{\mathbb{Z}}
\newcommand{\R}{\mathbb{R}}
\newcommand{\NN}{\mathbb{N}}
\newcommand{\CC}{\mathbb{C}}

\newcommand{\Bcal}{\mathcal{B}}
\newcommand{\Ccal}{\mathcal{C}}
\newcommand{\Acal}{\mathcal{A}}

\renewcommand{\Im}{\operatorname{Im}}
\newcommand{\im}{{\rm i}}

\newcommand{\Ocal}{\mathcal{O}}
\newcommand{\MB}{{\mathcal M}}
\newcommand{\erf}{\, \mathrm{erf}}
\newcommand{\sh}{{\rm sh}}

\newtheorem{theorem}{Theorem}[section]
\newtheorem{cor}[theorem]{Corollary}
\newtheorem{prop}[theorem]{Proposition}
\newtheorem{lemma}[theorem]{Lemma}

\newtheorem{remark}[theorem]{Remark}
\newtheorem{criterion}{Criterion}

\DeclareFontFamily{OT1}{rsfs}{}
\DeclareFontShape{OT1}{rsfs}{m}{n}{ <-7> rsfs5 <7-10> rsfs7 <10-> rsfs10}{}
\DeclareMathAlphabet{\mycal}{OT1}{rsfs}{m}{n}

\def\Ups{\Upsilon}

\begin{document}

\title[ ]{\textbf{ Linear instability of the Prandtl equations via hypergeometric functions and the harmonic oscillator}}

\author[
    F. De Anna\quad 
    J. Kortum
]{
    Francesco De Anna$^1$ \quad
    Joshua Kortum$^2$
}

 \address{
 $\,^{1,\,2}$ Institute of Mathematics, University of W\"urzburg 
 \\
 \medskip
 $\,^1$francesco.deanna@uni-wuerzburg.de\\
 $\,^2$joshua.kortum@uni-wuerzburg.de
}

\begin{abstract}

We establish a deep connection between the Prandtl equations linearised around a quadratic shear flow, confluent hypergeometric functions of the first kind, and the Schrödinger operator. \\
Our first result concerns an ODE and a spectral condition derived in \cite{MR2601044}, 
associated with unstable quasi-eigenmodes of the Prandtl equations. We entirely determine the space of solutions in terms of Kummer's functions. By classifying their asymptotic behaviour, we verify that the spectral condition has  a unique, explicitly determined pair of  eigenvalue and eigenfunction, the latter being expressible as a combination of elementary functions.\\
Secondly, we prove that any quasi-eigenmode solution of the linearised Prandtl equations around a quadratic shear flow can be explicitly determined from algebraic eigenfunctions of the Schrödinger operator with  quadratic potential.\\
We show finally that the obtained analytical formulation of the velocity align with previous numerical simulations in the literature.

\end{abstract}

\maketitle	


\section{Introduction}
\noindent 
This paper has two main objectives, aiming to establish a deep analytical connection between a family of solutions to the  Prandtl equations linearised around a parabolic shear layer, hypergeometric functions, and algebraic eigenfunctions of the harmonic oscillator.

\noindent 
The first objective is to address a criterion for an ordinary differential equation (cf.~\Cref{crit:spectral-condition-W})  introduced in \cite{MR2601044}, which asymptotically characterizes instabilities near points of loss of monotonicity of the flow.
Our first result states that all solutions of this particular ODE can be explicitly expressed in terms of Kummer's hypergeometric functions (cf.~\Cref{thm:explicit-form-of-X} and \Cref{cor:main-result-tau-W-unique}).

\noindent 
The second objective is to further develop these relations directly to the following linearised Prandtl system:
\begin{equation}\label{eq:linearised-Prandtl}
    \begin{cases}
    \partial_t u  + U_{\rm sh} \partial_x u   +  v  U_{\rm sh}'  - \partial_y^2 u   = 0,
    \quad 
    &(t,x,y) \in (0,T)\times \mathbb T \times \mathbb R_+,\\
    \partial_x u   + \partial_y v  = 0,
    &
    (t,x,y) \in (0,T)\times \mathbb T \times \mathbb R_+,
    \end{cases}
\end{equation}
where $(u,v)=\left(u(t,x,y),v(t,x,y)\right)$ is the velocity field, while the shear flow $U_{\rm sh} = U_{\rm sh}(y)$, with $y \in \mathbb{R}_+$, is considered as a parabolic shear layer characterized by
\begin{equation}\label{parabolic-shear-flow}
    U_{\rm sh}(y) = \alpha + \beta (y-a)^2,\quad 
    \text{for given } 
    a \geq 0,\;\alpha \in \mathbb R 
    \;
    \text{and}
    \; 
    \beta <0.
\end{equation}
We  determine  explicitly all solutions of \eqref{eq:linearised-Prandtl} that extend to complex values and can be written in the form
\begin{equation}\label{eq:intro-form-of-u-v-modes}  
    u(t,x,y) = u_k(y) \exp \left( \im k x + \sigma t \sqrt{|k|} \right),  \qquad 
    v(t,x,y) = v_k(y) \exp \left( \im k x + \sigma t \sqrt{|k|} \right),  
\end{equation}  
for a general  frequency $  k \in \mathbb{Z} \setminus \{0\} $ and an arbitrary $\sigma \in \mathbb{C}$.  
Our second result (cf.~\Cref{sec:intro-harmonic-osc} and \Cref{thm:TheoremSeparatedSolutions}) states that the functions  $u_k, v_k: \mathbb{R}_+ \to \mathbb{C}$ can be  written in terms of hypergeometric functions and vice versa suitable algebraic eigenfunctions of the Schr\"odinger operator. Additionally, we identify those that satisfy the no-slip boundary conditions at \( y = 0 \):   
\begin{equation}\label{intro:no-slip-boundary-conditions}  
    u_{|y=0}  = v_{|y=0}  = 0, \qquad (t,x) \in (0,T)\times \mathbb{T}.  
\end{equation}

\noindent 
The paper is structured as follows: we state our main results in \Cref{sec:crit-by-kum-fcts} and \Cref{sec:intro-harmonic-osc}, and we prove them in \Cref{sec:explicit-solutions-of-Y-and-X} and \Cref{sec:representation-via-harmonic-oscillator}, respectively. In \Cref{rmk:method-to-build-solutions}, we introduce a method to explicitly compute solutions of the form \eqref{eq:intro-form-of-u-v-modes} in terms of eigenfunctions of the Schrödinger operator, which we further develop through examples in \Cref{sec:application-of-the-method}. Finally, in \Cref{sec:explicit-form-of-shear-layer}, we apply this method to compute unstable (quasi-)modes of the Prandtl equations, expressing with analytical precision (cf.~\eqref{eq:V-Gerard-Dormy-explicit}), the so-called ``shear-layer corrector'' of \cite{MR2601044}.

\subsection{The criterion of unstable modes by Kummer's functions}\label{sec:crit-by-kum-fcts}$\,$

\smallskip 
\noindent 
Instabilities in boundary layers often arise from flow separations, a phenomenon in which the velocity field loses monotonicity and reverses direction.
The distinction between monotonic and non-monotonic behaviour is particularly evident in the analysis of the Prandtl equations, where well-posedness in Sobolev spaces is locally ensured by monotonicity \cite{oleinik1963prandtl,MW2015,MR3765768,AWXY2015}, while  significantly stronger regularities are required for non-monotonic profiles (analyticity and Gevrey classes \cite{MR3925144,LY2020,lombardo2003well,MR3429469,MR4028516,MR4635888,MR3855356,MR4465902}).

\smallskip  
\noindent  
For non-monotonic flows, unstable modes emerge already in the linearised equations \eqref{eq:linearised-Prandtl} around shear flows \( U_\sh(y) \)  exhibiting a critical point \( U_\sh'(a) = 0 \) which is non-degenerate (\( U_\sh''(a) \neq 0 \)). Under these conditions, the linear equations \eqref{eq:linearised-Prandtl} are ill-posed in Sobolev spaces \cite{MR2601044}, with solutions that exhibit rapid norm inflation in time  (cf.~\cite{MR2849481} for the nonlinear case and \cite{MR2952715} for initial profiles for which no solution exists). In \cite{MR2601044}, in particular, unstable modes in the tangential frequencies are built relying on a perturbative argument of the following spectral condition for a reduced ODE (cf.~(1.7) in \cite{MR2601044}):  
\begin{criterion}\label{crit:spectral-condition-W}
There exist $\tau \in \mathbb C$ with $\Im \tau < 0 $ and a smooth solution $W: \mathbb R \to \mathbb C$ of  
\begin{equation}\label{eq:W}
    (\tau - z^2)^2  W'(z) + \im \frac{\dd^3}{\dd z^3}\bigg[ (\tau - z^2) W(z)\bigg] = 0,
\end{equation}
such that $\lim\limits_{z \to -\infty}W(z) = 0$ and $\lim\limits_{z \to +\infty}W(z) = 1$. 
\end{criterion}
\noindent 
In essence, the variable $z \in \mathbb{R}$ is a rescaled version of the normal coordinate $y \in \mathbb{R}_+$ from \eqref{eq:linearised-Prandtl}, given by $z = \sqrt[4]{|k||\beta|}(y-a)$. It is centered at the critical point $y = a$ and describes regions that become increasingly localised around $a$ as the tangential frequencies $k \to +\infty$ (cf.~also \Cref{sec:equivalent-forms-of-the-linearised-Prandtl-equations} and \eqref{eq:equation-of-W-intro} for a detailed discussion within our parabolic shear flow \eqref{parabolic-shear-flow}).  
The function $(\tau - z^2)W(z)$ approximates $v_k(y)$ near $y = a$ (up to a multiplicative constant), while $\sigma \sim \im \tau$ in \eqref{eq:intro-form-of-u-v-modes}, at least asymptotically as $k \to +\infty$.

\smallskip
\noindent 
G\'erard-Varet and Dormy \cite{MR2601044} implicitly determined a couple ($W$, $\tau$) satisfying \Cref{crit:spectral-condition-W} through an auxiliary eigenvalue problem. Based on this result, they constructed solutions of the linearised Prandtl's system that, for large tangential frequencies $k \gg 1$, grow as $k^m e^{\im \tau \sqrt{k} t}$, for a positive $m>0$, over short $k$-dependent time scales (at least up to $t \sim \ln(k)/\sqrt{k}$, as  shown in \cite{MR4709568}), leading to rapid norm inflation in Sobolev spaces. \Cref{crit:spectral-condition-W} has moreover played a major role in additional instabilities in boundary layers (cf.~for instance Section 4 in \cite{MR2952715}, Lemma 1.2 in \cite{MR2876507} and (2.11) in \cite{MR3670620}).

\smallskip \noindent 
Our first result simplifies the statement of \Cref{crit:spectral-condition-W} by showing that, for any $\tau \in \mathbb{C}$, the ODE \eqref{eq:W} can be solved explicitly in terms of hypergeometric functions. Moreover, we establish that the function $W$ and the parameter $\tau$ in \Cref{crit:spectral-condition-W} are uniquely determined and can be expressed in terms of elementary functions, such as the Gauss error function (cf.~\Cref{cor:main-result-tau-W-unique}).

\smallskip \noindent 
Hypergeometric functions play a crucial role in fluid dynamics and turbulence studies. For instance, Bertolotti \cite{Bertolotti2000} established a connection between cross-flow vortices and attachment-line instabilities by representing the velocity and pressure fields in swept Hiemenz flows using Kummer and Tricomi functions. More recently, Salin and Talon \cite{Salin_Talon_2019} analysed core-annular flows, identifying unstable solutions for the axisymmetric linearized equations of each fluid using Bessel and confluent hypergeometric functions. See also \cite{Drazin_Reid_2004} for applications in Rayleigh's stability equation.

\noindent To the best of our knowledge, however, the use of hypergeometric functions in solving Prandtl-type equations has not yet been analytically established. Our first result states that the derivative $X(z) = W'(z)$ of any solution $W$ of equation \eqref{eq:W}, without prescribed boundary conditions,  can be expressed using the Kummer's confluent function $\mathcal M$:
\begin{equation}\label{def:kummer's-function}
    \mathcal M({\rm a} , {\rm c} , \zeta) = 
    \sum_{n = 0}^\infty \frac{({\rm a} )_n}{({\rm c} )_n}\frac{\zeta^n}{n!}, \qquad 
    \text{for any } \zeta \in \mathbb C,
\end{equation}
where ${\rm a} \in \mathbb C$, ${\rm c}  \in \mathbb C \setminus \{0, -1, -2, \dots \}$ and $(\cdot )_n$ stands for the Pochhammer symbol.
\noindent 
For any $\tau \in \mathbb C$, we set $\mathcal A_\tau = \{ z \in \mathbb C  \text{ such that } z^2 = \tau\}$ and we determine in terms of $\mathcal M$ all solutions $X(z) = W'(z)$ of the following complex ordinary differential equation:
\begin{equation}\label{eq:X-intro}
        \im (\tau -z^2)X''(z) -6 \im  z X'(z) + ((\tau -z^2)^2-6 \im)X(z)= 0,\qquad 
        z \in \mathbb C \setminus \mathcal{A}_\tau.
\end{equation}
\begin{theorem}\label{thm:explicit-form-of-X}
    Let $\tau \in \mathbb{C}$ be an arbitrary complex number and define the constants ${\bf a}_\tau, {\bf b}_\tau, {\bf c}_\tau, {\bf d}_\tau \in \mathbb{C}$ as  
    \begin{equation*}
       {\bf a}_\tau :=  -\frac{1}{4}
                \Big(1 + \tau 
                 e^{\frac{7\pi}{4}\im }
                \Big)
        ,
        \quad 
        {\bf b}_{\tau}:= 
         \frac{1}{4}
                \Big(3 - \tau 
                 e^{\frac{7\pi}{4}\im }
                \Big)
        ,
        \quad 
        {\bf c}_{\tau}:= 
         \frac{1}{4}
                \Big(1 - \tau 
                 e^{\frac{7\pi}{4}\im }
                \Big),
        \quad 
        {\bf d}_{\tau}:= 
         \frac{1}{4}
                \Big(5 - \tau 
                 e^{\frac{7\pi}{4}\im }
                \Big).
    \end{equation*} 
    Every solution $X: \mathbb{C} \setminus A_\tau \to \mathbb{C}$ of \eqref{eq:X-intro} can be written explicitely as  $X(z) = c_1 X_{\tau, 1}( z) + c_2 X_{\tau, 2}(z) $, where $c_1, c_2 \in \mathbb{C}$ are arbitrary constants, and the functions $X_{\tau,1}$ and $X_{\tau,2}$ are meromorphic and defined by
    \begin{equation}\label{eq:main-thm-X1-X2}
    \begin{alignedat}{32}
        X_{\tau, 1}(z) &= 
        \frac{\exp
        \left( \frac 12 e^{\frac{3\pi}{4}\im } z^2\right)
        }{(\tau-z^2)^2}
        &&\Bigg[ 
            \tau\, 
            \mathcal M
            \left(
                {\bf{a}}_\tau,
                \,
                \frac{1}{2}, 
                \,
                 e^{\frac{7\pi}{4}\im }
                 z^2
            \right)
            &&&&-
            4{\bf a}_\tau  z^2 \,
            &&&&&&&&\mathcal M
            \left(
               {\bf b}_{\tau}, 
                \frac{3}{2}, 
                 e^{\frac{7\pi}{4}\im }
                 z^2
            \right)
        &&&&&&&&&&&&&&&&\Bigg],
        \\
        X_{\tau, 2}(z) 
        &= 
        z
        \frac{\exp
        \left( \frac 12 e^{\frac{3\pi}{4}\im } z^2\right)
        }{(\tau-z^2)^2}
        &&\Bigg[ 
            \;\mathcal M
            \left(
                {\bf c}_{\tau},
                \,
                \frac{3}{2}, 
                \,
                e^{\frac{7\pi}{4}\im }z^2
            \right)
            &&&&+  
            e^{\frac{7\pi}{4}\im } 
            \frac{z^2}{3}
            &&&&&&&&\mathcal M
            \left(
                {\bf d}_{\tau}, 
                \,
                \frac{5}{2}, 
                \,
                 e^{\frac{7\pi}{4}\im } z^2
            \right)
        &&&&&&&&&&&&&&&&\Bigg].
    \end{alignedat}    
    \end{equation}
\end{theorem}
\noindent 
The proof of \Cref{thm:explicit-form-of-X} is given in \Cref{sec:explicit-solutions-of-Y-and-X}, however some remarks on its implications are already in order. There is no restriction on the choice of $\tau \in \mathbb C$ and the functions $X_{\tau,1}$ and $X_{\tau,2}$ are even and odd, respectively, ensuring their linear independence. Moreover, the prefactor $\exp (\frac 12 e^{\frac{3\pi}{4}\im } z^2 )$  decays exponentially  for real $z \in \mathbb{R}$. Meanwhile, the Kummer's functions in \eqref{eq:main-thm-X1-X2} have a well-known asymptotic expansion at $z \to \pm \infty$ (cf.~Chapter~7.10 in \cite{OLVER1974229}), typically of exponential growth, except when they collapse to polynomials.
This corresponds mainly to the discrete values of ${\bf a}_\tau\,{\bf b}_\tau \in -\mathbb{N}_0$ (if $c_1 \neq 0$) and ${\bf c}_\tau, {\bf d}_\tau \in -\mathbb{N}_0$ (if $c_2 \neq 0$). 

\smallskip
\noindent 
These remarks highlight the key impact of \Cref{thm:explicit-form-of-X}, namely it determines a pair $(\tau, W)$ that satisfies all conditions of \Cref{crit:spectral-condition-W}.
 Setting $\tau = e^{\frac{5\pi}{4}\im } =-(1+\im)/{\sqrt{2}}$ yields ${\bf a}_\tau = 0$, ${\bf b}_\tau = 1$, ${\bf c}_\tau = 1/2$, and ${\bf d}_\tau = 3/2$. Consequently, the first Kummer function in $X_{1, \tau}$ simplifies to
$\mathcal{M}({\bf{a}}_\tau, 1/2, e^{\frac{7\pi}{4}\im }  z^2) = 1$, while ${\bf a}_\tau \mathcal{M}({\bf{b}}_\tau, 3/2, e^{\frac{7\pi}{4}\im }  z^2)$ vanishes due to ${\bf a}_\tau = 0$. Choosing $c_1 \in \mathbb{C} \setminus \{0\}$ and $c_2 = 0$, the corresponding solution $X$ exhibits exponential decay as follows:
\begin{equation}\label{intro:X-of-criterium}
    \tau = e^{\frac{ 5\pi}{4}\im  },
    \qquad 
    X(z) =  
        c_1e^{\frac{5\pi}{4}\im }
        \frac{\exp
        \left( \frac 12 e^{\frac{3\pi}{4}\im } z^2\right)
        }{\big( e^{\frac{5 \pi}{4}\im}-z^2\big)^2}
        = 
        c_1
        e^{\frac{5\pi}{4}\im }
        \frac{\exp
        \left( -\frac 12 e^{-\frac{\pi \im}{4} } z^2\right)
        }{\big( e^{\frac{\pi }{4}\im} + z^2\big)^2}
        = 
        c_1
         e^{\frac{ 3\pi }{4}\im}
        \frac{\exp\left(  - \frac{\eta(z)^2}{2}\right) 
        }{\big(1+\eta(z)^2\big)^2},
\end{equation}
with $\eta(z) = e^{-\frac{\pi \im}{8}}z$. 
Since ${\rm Im}\, \tau = -1/\sqrt{2} < 0$, \Cref{crit:spectral-condition-W} is governed by a primitive function $W$ of $X$ in \eqref{intro:X-of-criterium}, with a suitable $c_1\neq 0$ that ensures the correct limits as $z \to \pm \infty$. The function $W$ is stated in the following corollary.
\begin{cor}\label{cor:main-result-tau-W-unique}
 There exists a unique $\tau \in \mathbb{C}$ with $\Im \tau < 0$ and a unique function $W = W(z)$ satisfying the conditions of \Cref{crit:spectral-condition-W}, given by
 \begin{equation*}
     \tau =  e^{\frac{5\pi}{4}\im }
     =-\frac{1+\im}{\sqrt{2}},
     \qquad 
     W(z )
     = 
     \frac{1}{2}
     \left(
        1
        +
        \erf
        \left( 
            \frac{\eta(z)}{\sqrt{2}}
        \right)
     +
     \sqrt{\frac{2}{\pi}}
        \frac{\eta(z)}{ 1+\eta(z)^2}
        e^{
            - \frac{\eta(z)^2}{2}
        }
        \right)
    \quad
    \text{with}
    \quad
    \eta(z) = e^{-\frac{\im \pi}{8}}z,
\end{equation*}
where $\erf(\zeta) =\frac{2}{\sqrt{\pi}} \int_0^\zeta e^{-\omega^2} \dd \omega $ is the  Gauss error function in $\zeta \in \mathbb C$. 
\end{cor}
\noindent 
In this paper, we integrate only entire functions with the definite integrals understood as any path integrals. The existence part of \Cref{cor:main-result-tau-W-unique} follows from setting $c_1 = e^{\frac{\pi \im}{8}}\sqrt{2/\pi} \neq 0$ in \eqref{intro:X-of-criterium}, along with the following identity and  limits of the error function (cf.~the asymptotic expansion 7.12.1 in \cite{Temme2010}):
\begin{equation*}
    \frac{\dd}{\dd\eta}
    \left(
        \frac{1}{2}
        \erf
        \left( 
            \frac{\eta}{\sqrt{2}}
        \right)
        +
        \frac{1}{\sqrt{2\pi}}
        \frac{\eta 
        e^{ - \frac{\eta^2}{2}} }{ 1+\eta^2}
    \right)
    =
   \sqrt{\frac{2}{\pi}}
   \frac{e^{ - \frac{\eta^2}{2}}}{( 1+\eta^2)^2},
   \quad 
   \lim_{\substack{\zeta \to \infty_{\mathbb{C}} \\ |\arg(\zeta) |< \frac{\pi}{4}-\delta }} 
    \erf 
    ( \zeta) =  1,
   \quad 
   \lim_{\substack{\zeta \to \infty_{\mathbb{C}} \\ |\arg(-\zeta) |< \frac{\pi}{4} -\delta }} 
    \erf 
    ( \zeta) = -1,
\end{equation*}
for any $\delta >0$. The uniqueness in \Cref{cor:main-result-tau-W-unique}, however, requires a detailed analysis of the asymptotic expansions of the solutions \eqref{eq:main-thm-X1-X2} for general $\tau \in \mathbb C$, with ${\rm Im}\, \tau <0$.  We briefly postpone its explanation to \Cref{cor:GrowthOnR}.
\begin{remark}  
  Using a shooting method, G\'erard-Varet and Dormy numerically approximated $\tau \approx -0.706 - 0.706 \im$ (cf.~Section 5.1 in \cite{MR2601044}). This value is close to $\tau = e^{\frac{5\pi}{4} \im} \approx -0.707107 - 0.707107 \im$, as established in \Cref{cor:main-result-tau-W-unique}. They also constructed their instability mechanism using the so-called ``shear-layer'' velocity, defined in terms of $W$. Since $W$ is now explicit, we can formulate this shear layer in terms of elementary functions. Details are provided in \Cref{sec:explicit-form-of-shear-layer}, and the corresponding plot (cf.~\Cref{fig:Plotphik3} at the end of this paper) matches their numerical simulation (cf.~Figure 3, page 607 in \cite{MR2601044}).  
\end{remark}  
\noindent 
 As we describe in the next section, \Cref{thm:explicit-form-of-X} takes a more natural form under the parabolic shear flow $U_{\rm sh}$ introduced in \eqref{parabolic-shear-flow}. In this setting, equation \eqref{eq:W} exactly corresponds  to the modes of the Prandtl system \eqref{eq:linearised-Prandtl} at any positive frequency $k \in \mathbb{N}$, not just asymptotically (cf.~\Cref{sec:equivalent-forms-of-the-linearised-Prandtl-equations}). This perspective provides a direct method for constructing solutions of the Prandtl system \eqref{eq:linearised-Prandtl} in terms of Kummer's functions.

\subsection{Explicit solutions of the Prandtl equations via harmonic oscillator}
\label{sec:intro-harmonic-osc}$\,$

\medskip 
\noindent 
From \Cref{thm:explicit-form-of-X} and \eqref{eq:main-thm-X1-X2}, we observe that there exists a sequence of $\tau \in \mathbb{C}$ that generates solutions $X$ decaying exponentially as $z \to \pm \infty$: 
\begin{equation*}
\tau = e^{\frac{\pi \im }{4}}(2n-1),\quad n \in \mathbb N_0\setminus\{1\}
\Rightarrow \left\{
\begin{alignedat}{8}
    X = X_{1,\tau}
    \text{ with }
    &{\bf a}_{\tau} = -\frac{n}{2} 
    \text{ and }{\bf b}_{\tau} = 1-\frac{n}{2}   \quad &&&&\text{if } n \text{ is even},\\
    X = X_{2,\tau}
    \text{ with }
    &{\bf c}_{\tau} = \frac{1-n}{2} 
    \text{ and }{\bf d}_{\tau} =
    \frac{3-n}{2}\quad &&&&\text{if } n\text{ is odd}.
\end{alignedat}
\right.
\end{equation*}
The case $n = 0$ was described in \eqref{intro:X-of-criterium}. Since $\tau$ plays a role similar to an eigenvalue, this behaviour bears a resemblance to the bound state spectrum of the Schr\"odinger operator and the quantization of energy levels. Our second result aims to demonstrate that this correlation is indeed not accidental.

\smallskip 
\noindent 
As further intuition, we draw inspiration from the derivation of the Orr-Sommerfeld equation (e.g.~\cite{schlichting2016boundary}) and recent studies on the well-posedness of the Prandtl equations in Gevrey classes \cite{MR3925144}. Defining $\pm = {\rm sgn} (k)$ and introducing the following variable and parameter equivalences (cf.~\Cref{sec:equivalent-forms-of-the-linearised-Prandtl-equations} for all details):  
\begin{alignat*}{8}
    &y \in \mathbb R_+ 
    \quad 
    &&\leftrightarrow 
    \quad
    z =  \sqrt[4]{ |\beta||k|}(y-a)\in \mathbb R
    \quad 
    &&&&\leftrightarrow 
    \quad
    \eta = e^{\mp  \frac{\pi}{8}\im } z \in \mathbb C\\
    &\sigma \in \mathbb C 
    \quad 
    &&\leftrightarrow 
    \quad
      \tau =- \frac{\im \sigma}{\sqrt{|\beta|}}  
      \pm  \alpha\sqrt{ \frac{ |k| }{|\beta|}}  \in  \mathbb C   
    \quad 
    &&&&\leftrightarrow 
    \quad
    \mu =  -\im \tau e^{\pm \frac{\pi \im}{4} } 
          \in  \mathbb C,
\end{alignat*}
the Prandtl eigenproblem, when expressed in the rotated variable $\eta \in \mathbb{C}$, takes the form (cf.~\eqref{eq:equation-of-Upsilon-intro}):  
\begin{equation}\label{eq:intro-Upsilon-equation}
    - \Ups'''(\eta)+ (\eta^2-\mu)\Ups'(\eta)-2\eta\Ups(\eta)=0 \qquad \eta \in \CC,
\end{equation}
where $\Upsilon(\eta)$ corresponds to $v_k(y)$, up to a multiplicative constant. If the last minus sign were positive, the resulting equation would correspond to the derivative of the Schr\"odinger equation with a quadratic potential \( V(\eta) = \eta^2 \). This observation motivates the introduction of the operator 
\begin{align*}
    \Bcal_\mu = - \frac{\dd^2 }{\dd \eta^2} + V(\eta) -\mu
    \quad
    \Rightarrow
    \quad 
    \frac{\dd}{\dd\eta} \Big(\Bcal_\mu
    \Ups(\eta)\Big)  = 4 \eta  \Ups(\eta)
\end{align*}
and heuristically a new quantity $\Phi=\Phi(\eta)$ defined by $\Bcal_\mu \Phi= 2\Ups$. In resemblance to \cite{MR3925144} and \cite{schlichting2016boundary}, a notable simplification arises when considering the commutator of  the derivative and the operator $\Bcal_\mu^2$:
\begin{equation*}
    \frac{\dd}{\dd\eta} \Bcal_\mu^2 
    =
    \Bcal_\mu^2\frac{\dd}{\dd\eta}-4\frac{\dd}{\dd\eta}  + 4\eta \Bcal_\mu 
    \quad 
    \Rightarrow
    \quad \left[\frac{\dd}{\dd\eta}, \Bcal_\mu^2 \right]\Phi = 4\eta \Bcal_\mu\Phi - 4 \Phi'.
\end{equation*}
Setting $\psi(\eta) = \Phi'(\eta)$, we therefore obtain that the $\Ups$-equation reduces to the following identity
\begin{align}  \label{eq:SchrodingerSquared}
    \Bcal_\mu^2 \psi - 4 \psi = 0.
\end{align}
Since $\Bcal_\mu = \Bcal_{\mu-2} -2 $ and $\Bcal_\mu = \Bcal_{\mu+2} +2 $, one has that $\Bcal_\mu^2 = (\Bcal_{\mu-2} -2)(\Bcal_{\mu+2} +2) =\Bcal_{\mu-2}\Bcal_{\mu+2} + 4$, hence we can factorise the latter equation into 
\begin{align*} 
    \Bcal_{\mu-2} \Bcal_{\mu+2} \psi =0.
\end{align*}
Both operators  $\Bcal_{\mu+2}$ and $\Bcal_{\mu-2}$ commute and leave us investigating one of the two kernels, for instance 
\begin{align*}
     \Bcal_{\mu+2} \psi
     = 
     - \psi'' + \eta^2 \psi  -(\mu+2) \psi  = 0,
\end{align*}
which is indeed the Schr\"odinger equation. Moreover,  any solution $\psi$, $\Upsilon$ (and thus $v_k$) is then determined by the integral (c.f.\ Proposition \ref{prop:RepresentationByPsi})
\begin{equation}\label{integral-operator-for-Ups}
\begin{aligned}
    \Ups(\eta) 
    &= \frac{1}{2}\Bcal_{\mu}
    \left[ 
        \int_{\eta_*}^\eta \psi(\xi)d\xi 
    \right] \\
    &= 
    \int_{\eta_*}^\eta\left(1+ \frac{\eta^2-\xi^2}{2}\right)\psi(\xi) d\xi 
    -
    \frac{\psi'(\eta_*)}{2},
\end{aligned}    
\end{equation}
for a constant $\eta_*\in \mathbb C$ that will set the boundary conditions (cf.~\Cref{thm:TheoremSeparatedSolutions}). 
At this point, we want to clarify that the shift $\mu\mapsto \mu +2$ of the associated ``energy'' spectrum  is creating instability. While all (functional analytic) eigenvalues of the quantum harmonic oscillator are positive odd numbers, the smallest value here is $\mu=1-2=-1$ leading to a ``negative energy'' and, in consequence, to an unstable eigenmode.

\smallskip
\noindent 
We aim to make the latter Ansatz explicit employing Kummer's functions. To this end we first recall some result of the harmonic oscillator. For a given constant $\mu \in \mathbb{C}$, the general solution $\psi$ of the equation  
\begin{equation} \label{eq:schrödingerEquation}
    -\psi''(\eta) + \eta^2 \psi (\eta) = (\mu+2) \psi (\eta), \qquad \eta \in \mathbb{C},
\end{equation}
is given by $\psi = c_{1} \psi_{\mu,1} + c_2 \psi_{\mu,2}$ for arbitrary constants $c_1, c_2 \in \mathbb{C}$ and functions $\psi_{\mu,1}$ and $\psi_{\mu,2}$ defined in terms of the following Kummer's functions (although a classical result, we provide a short proof in \Cref{lemma:HarmonicOscillatorSolutions}):  
    \begin{equation}\label{eq:main-thm-harmonic-oscillator}
    \begin{alignedat}{8}
        \psi_{\mu,1} (\eta) 
        &:=
        \mathcal M \Big( -\frac{1+\mu}{4}  , \frac{1}{2}, \eta^2 \Big)
        \exp\Big( -\frac{\eta^2}{2}\Big),
        \\
        \psi_{\mu,2} (\eta)
        &:= 
        \eta   \,
        \mathcal M \Big(\; \frac{1-\mu}{4}, \frac{3}{2}, \eta^2 \, \Big)
        \exp\Big( -\frac{\eta^2}{2}\Big). 
    \end{alignedat}
    \end{equation}    
    For a given $\eta_* \in \mathbb C$, we define the function $\Upsilon_{\mu, \eta_*,1}:\mathbb C \to \mathbb C$ through the operator \eqref{integral-operator-for-Ups} applied to $\psi = \psi_{\mu,1}$:
    \begin{equation}\label{eq:Upsilon1}
        \Upsilon_{\mu, \eta_*,1}(\eta) 
        := 
        \int_{\eta_*}^\eta\bigg( 1 +  \frac{\eta^2 - \xi^2}{2}\bigg) \psi_{\mu,1}(\xi) d\xi   
        -
        \frac{\psi_{\mu,1}'(\eta_*)}{2}.
    \end{equation}
    Similarly, we define $\Upsilon_{\mu, \eta_*,2}:\mathbb C \to \mathbb C$  considering two separate cases $\mu \in \mathbb C\setminus\{1\}$ and $\mu = 1$ (cf.~also \Cref{rmk:intro-mu=1}):
    \begin{equation} \label{eq:Upsilon2}
        \Upsilon_{\mu, \eta_*,2}(\eta) 
        := 
        \begin{cases}
        \int_{\eta_*}^\eta\big( 1 +  \frac{\eta^2 - \xi^2}{2}\big) \psi_{\mu,2}(\xi) d\xi   
        -
        \frac{\psi_{\mu,2}'(\eta_*)}{2}
        \qquad 
        &\text{if }\mu \neq 1,\\
        \int_{\eta_*}^\eta\big( 1 +  \frac{\eta^2 - \xi^2}{2}\big) g(\xi) d\xi   
        -
        \frac{g'(\eta_*)}{2}
        ~\text{ with }g(\xi) = e^{\frac{\xi^2}{2}}\erf(\xi)
        &\text{if }\mu =1.
        \end{cases} 
\end{equation}
Our second result states that the normal component $v_k$ of the velocity field in \eqref{eq:intro-form-of-u-v-modes} is a linear combination of $\Upsilon_{\mu, \eta_*,1} $, $\Upsilon_{\mu, \eta_*,2}$ and a quadratic function, provided the variable $\eta \in \mathbb{C}$ is appropriately related to $y \in \mathbb{R}_+$ and the constant $\mu \in \mathbb{C}$ is linked to $\sigma \in \mathbb{C}$ in \eqref{eq:intro-form-of-u-v-modes}.
\begin{theorem} \label{thm:TheoremSeparatedSolutions}
     Let $U_\sh(y)= \alpha + \beta(y-a)^2$  be as in \eqref{parabolic-shear-flow} and consider $k \in \mathbb Z \setminus\{0\}$ and $\sigma \in \mathbb C $. Denote $\pm = {\rm sgn} (k) $, $\mp = -{\rm sgn} (k) $, and define the constant $\mu = \mu(k, \sigma, \alpha, \beta) \in \mathbb C$ and  $\eta: \mathbb R_+ \to \mathbb C$
     as
    \begin{equation}\label{eq:main-thm-mu}
    \begin{aligned}
        \mu &:= 
        -\frac{\sigma}{\sqrt{|\beta|}} e^{ \pm \frac{\pi \im }{4}}
         +
         \frac{ \alpha \sqrt{|k|}}{\sqrt{|\beta|}} e^{ \mp \frac{\pi \im }{4}}
        \in \mathbb C,
        \qquad
         \eta(y):= e^{\mp  \frac{\pi}{8}\im } \sqrt[4]{|\beta||k|}(y-a)\in \mathbb C.
    \end{aligned}
    \end{equation}
    Then $ u(t,x,y) = \phi_k'(y)e^{ \im  k x +  \sigma \sqrt{|k|} t } $  and $ v(t,x,y) = -\im k \phi_k(y) e^{ \im  k x +  \sigma \sqrt{|k|} t } $  
    generate a smooth solution of Equation~\eqref{eq:linearised-Prandtl} (without boundary conditions) if and only if $\phi_k$ satisfies
    \begin{equation}\label{thm:main-relation-on-uk}
    \begin{alignedat}{8}
        \phi_k(y) 
        &= 
        c_0 \big(\,\mu-\eta(y)^2\,\big)
        &&
        \;+\; 
        c_1
        \big(\Upsilon_{\mu, \eta_*,1}
        \circ \eta
        \big) (y)
        &&&&
        \;+\; 
        c_2 \,
         \big(\Upsilon_{\mu, \eta_*,2}
        \circ \eta
        \big) (y)
        \qquad 
        \forall y \in \mathbb R_+,
    \end{alignedat}    
    \end{equation}  
    where $\eta_*, c_0, c_1, c_2 \in  \mathbb C$ are general constants and $\Upsilon_{\mu, \eta_*,1}$ and $\Upsilon_{\mu, \eta_*,2}$ are set in \eqref{eq:Upsilon1} and \eqref{eq:Upsilon2}, respectively. 

    \noindent
    Additionally, the no-slip boundary conditions $\phi_k(0) = \phi_k'(0) = 0$ are satisfied if  $\eta_* =  -e^{\mp \frac{\pi }{8}\im  }  a \sqrt[4]{|\beta||k|}$ and the triple $(c_0,\,c_1,\,c_2)\in \mathbb C^3$ is a solution of the linear system
   \begin{equation}\label{eq:main-thm-constants-for-bdycdt}
    \begin{pmatrix}
        \mu-\eta_*^2 & -\frac 12 \psi_{\mu,1}'(\eta_*) 
        &  -\frac 12 \psi_{\mu,2}'(\eta_*) \\[0.5em]
        -2\eta_* & \psi_{\mu,1}(\eta_*) & \psi_{\mu,2}(\eta_*)
    \end{pmatrix}
    \raisebox{-0.8em}{
    $
    \begin{pmatrix}
        c_1 \\[0.8em] c_2 \\[0.8em] c_3
    \end{pmatrix}
    $
    }
    =
    \begin{pmatrix}
        0 \\[0.5em] 0
    \end{pmatrix},\qquad \text{if } \mu \neq 1,
\end{equation}
while, in the case $\mu = 1$, $(c_0,\,c_1\,c_2)\in \mathbb C^3$ satisfies \eqref{eq:main-thm-constants-for-bdycdt} with $\psi_{\mu,2}(\eta)$ replaced by $g(\eta) = e^{\frac{\eta^2}{2}} \erf(\eta)$.
\end{theorem}
\begin{remark}\label{rmk:method-to-build-solutions}
\Cref{thm:TheoremSeparatedSolutions} provides a systematic approach to computing all solutions of System \eqref{eq:linearised-Prandtl} in the form \eqref{eq:intro-form-of-u-v-modes}, as well as selecting those that satisfy the no-slip boundary conditions \eqref{intro:no-slip-boundary-conditions} at $y=0$:  
\begin{enumerate}[(i)]
    \item Choose a general $\sigma \in \mathbb{C}$ and $k \in \mathbb{Z} \setminus \{0\}$ in \eqref{eq:intro-form-of-u-v-modes}, and compute the constant $\mu \in \mathbb{C}$ as defined in \eqref{eq:main-thm-mu}.  
    \item Determine the two algebraic eigenfunctions $\psi_{\mu,1}$ and $\psi_{\mu,2}$ of the Schr\"odinger operator in \eqref{eq:main-thm-harmonic-oscillator}. These correspond to bound state eigenfunctions if $\mu = 2n -1$ for some $n \in \mathbb{N}_0\setminus \{1 \}$ (where Kummer's functions reduce to polynomials, simplifying the next steps). 
    \item Set $\eta_* = -e^{\mp \frac{\pi}{8} \im } a \sqrt[4]{|\beta||k|}$ to enforce the no-slip boundary conditions. Otherwise, choose any $\eta_* \in \mathbb{C}$; calculations are simplified when setting $\eta_* = 0$.  
    \item Compute the functions $\Upsilon_{ \mu, \eta_*,1}$ and $\Upsilon_{\mu, \eta_*,2}$ explicitly, as defined in \eqref{eq:Upsilon1} and \eqref{eq:Upsilon2}, and define $u$ and $v$ according to the stream function $\phi_k$ in \eqref{thm:main-relation-on-uk}.  
    \item For solutions satisfying the no-slip boundary conditions, select $c_0$, $c_1$, and $c_2$ as given in \eqref{eq:main-thm-constants-for-bdycdt}; otherwise, these constants can be chosen freely.  
\end{enumerate}
Examples illustrating this method are provided in \Cref{sec:application-of-the-method}.
\end{remark}
\begin{remark}\label{rmk:intro-mu=1}
    The case $\mu = 1$ in \Cref{thm:TheoremSeparatedSolutions} is not particularly relevant for identifying unstable modes, as it forces $\sigma$ to have negative real part leading to stable modes. However, we can employ a separate analysis because $\psi_{1,2}(\eta) = \eta e^{-\frac{\eta^2}{2}}$ in \eqref{eq:main-thm-harmonic-oscillator}, and the corresponding function $\Ups(\eta)$ from \eqref{integral-operator-for-Ups} is then proportional to $1 - \eta^2 = \mu - \eta^2$.  

    \noindent 
    To obtain a linearly independent function $\Upsilon_{\mu, \eta_*,2}$ in \eqref{eq:Upsilon2}, we observe that any function in the kernel of either $\Bcal_{\mu+2}$ or $\Bcal_{\mu-2}$ can be chosen. For $\mu \neq 1$, we select $\Bcal_{\mu+2}$, while for $\mu = 1$, it is more convenient to consider $\Bcal_{\mu-2} = \Bcal_{-1}$. The function $g(\eta) = e^{\frac{\eta^2}{2}}\erf (\eta) $ lies in the kernel of $\Bcal_{-1}$ and is linearly independent of both $\psi_{1,1}(\eta) = e^{\frac{\eta^2}{2}} - \sqrt{\pi} \eta \erf(\eta)$ and $\psi_{1,2}(\eta) = \eta e^{-\frac{\eta^2}{2}}$.  
\end{remark}

\noindent 
We conclude our analysis with a result that transfers well-known asymptotics of Kummer's functions, as defined in \eqref{eq:main-thm-harmonic-oscillator}, to the functions $\Upsilon_{\mu, \eta_*,1}$ and $\Upsilon_{\mu, \eta_*,2}$. In particular, we state the following corollary.
\begin{cor}\label{cor:GrowthOnR}
Let $\mu \in \mathbb{C}$ be such that $\mu \neq 2n-1$ for any $n \in \mathbb{N}_0$, and let the constants $\eta_*, c_0, c_1, c_2 \in \mathbb{C}$ be arbitrary with $(c_1,c_2) \neq (0,0)$. Then, the function
\begin{equation}\label{eq:Upsilon-final-cor}
    \Upsilon(\eta) = 
    c_0(\mu- \eta^2) + 
    c_1 
    \Upsilon_{\mu, \eta_*,1}(\eta) +
    c_2
    \Upsilon_{\mu, \eta_*,2}(\eta)
\end{equation}
satisfies the following asymptotics:
\begin{equation*}
     \lim_{\substack{\eta \to \infty, \\ |\arg(\eta) |< \frac{\pi}{4} }} 
     \left|
        \frac{\Upsilon(\eta)}{\mu-\eta^2}
    \right|
    =
    +\infty 
    \qquad 
    \text{or}
    \qquad 
     \lim_{\substack{\eta \to \infty, \\ |\arg(-\eta) |< \frac{\pi}{4} }} 
     \left|
        \frac{\Upsilon(\eta)}{\mu-\eta^2}
    \right|
     =
     + \infty.
\end{equation*}
In case of $\mu=-1$, it holds
\begin{align*}
     \lim_{\substack{\eta \to \infty, \\ |\arg(\eta) |< \frac{\pi}{4} }} 
     \left|
        \frac{\Upsilon(\eta)}{\mu-\eta^2}
    \right|
    \neq 0
    \qquad 
    \text{or}
    \qquad 
     \lim_{\substack{\eta \to \infty, \\ |\arg(-\eta) |< \frac{\pi}{4} }} 
     \left|
        \frac{\Upsilon(\eta)}{\mu-\eta^2}
    \right|
     \neq 0.
\end{align*}
\end{cor}
\noindent
This result ultimately establishes the uniqueness of the function $W$ stated in \Cref{cor:main-result-tau-W-unique}. In \Cref{sec:equivalent-forms-of-the-linearised-Prandtl-equations}, indeed, we express any pair $(\tau,W)$ satisfying the ODE \eqref{eq:W} as  
 $(\tau - z^2) W(z) = \Upsilon(\eta(z))$, with $\eta(z) = e^{-\frac{\pi }{8}\im}z$, $\mu = e^{-\frac{\pi  }{4}\im} \tau$, and $\Upsilon$ satisfying the Prandtl eigenproblem \eqref{eq:intro-Upsilon-equation} (cf.~\Cref{remark:relation-between-Upsilon-and-W}). Due to \Cref{cor:GrowthOnR} and the limits $\lim\limits_{z \to -\infty}W(z) = 0$ and $\lim\limits_{z \to +\infty}W(z) = 1$, it follows that $\tau = e^{\frac{\pi }{4}\im}(2n-1)$ for some $n \in \mathbb{N}_0$. Furthermore, since ${\rm Im}\, \tau$ must be negative, the only remaining possibility is $n = 0$, leading to $\tau = e^{\frac{5\pi }{4}\im}$.

\section{Equivalent forms of the linearised Prandtl's System}\label{sec:equivalent-forms-of-the-linearised-Prandtl-equations}

\noindent 
In this section, we begin our analysis with a review of equivalent forms of the equations~\eqref{eq:linearised-Prandtl}. 
We briefly revise two transformations: the first rescales $y \in \mathbb{R}_+$ into a variable $z \in \mathbb{R}$ as a spatial scale in $k$ around $y = a$  (cf.~\eqref{eq:intro-def-of-z-and-F}); the second rotates $z \in \mathbb{R}$ to a complex variable $\eta \in \mathbb{C}$ (cf.~\eqref{eq:intro-def-of-eta-and-Upsilon}), yielding eventually to a third-order linear ODE, a Prandtl eigenproblem, where most coefficients are real (cf.~\eqref{eq:equation-of-Upsilon-intro}). The novelty of our work does not lie in the equations themselves (some of them already introduced in \cite{MR2601044}) but in their explicit resolution (starting from \Cref{sec:explicit-solutions-of-Y-and-X}). Additionally, since our approach is not asymptotic, we also track the corresponding boundary conditions in terms of both $z$ and $\eta$. For clarity, we summarize all these change of variables, new functions and parameters in \Cref{tab:variables_parameters_functions} (we recall that $\pm = {\rm sgn}(k)$).
%
%
%
%
%
\begin{table}[h]
    \centering
    \renewcommand{\arraystretch}{1.5} 
    \begin{tabular}{|c|c|c|}
        \hline
        \multicolumn{3}{|c|}{\textbf{Variables}} \\ 
        \hline 
        $y\in \mathbb R_+$   & 
        $z =  \sqrt[4]{ |\beta||k|}(y-a)\in \mathbb R$  & 
        $\eta = e^{\mp \frac{\pi}{8}\im }z \in   \mathbb R \cdot e^{\mp \frac{\pi}{8}\im }$ \\ 
        \hline 
        $y = 0$   & 
        $z_* =  -a \sqrt[4]{ |\beta||k|}  $  & 
        $\eta_* = -a  e^{\mp \frac{\pi}{8}\im } \sqrt[4]{ |\beta||k|}  $ \\ 
        \hline
        \multicolumn{3}{|c|}{\textbf{Parameters}} \\ 
        \hline
         $\sigma \in \mathbb C $   & 
         $\tau =  -  \frac{ \im \sigma}{\sqrt{|\beta|}} \pm \alpha \sqrt{\frac{|k|}{|\beta|}} \in  \mathbb C $  & 
          $\mu =  \tau e^{-\frac{\pi  }{2}\im\pm \frac{\pi }{4} \im } 
          \in  \mathbb C 
          $ \\ 
        \hline
        \multicolumn{3}{|c|}{\textbf{Functions}} \\ 
        \hline
        $\begin{aligned}
            u_k(y) &= \;\phi_k'(y)\\ 
            v_k(y) &= -\im k \phi_k(y)
        \end{aligned}$
        & 
        $ \begin{aligned}
        F_{\pm }(z) 
        &=    
        \phi_k \left(  a + \tfrac{ 1 }{ \sqrt[4]{ |\beta||k|}}z \right) \\
        &=(\tau \mp z^2 )W_{\pm }(z)
        \end{aligned}$  & 
        $ 
        \Upsilon(\eta) 
        =  F_{\pm } (e^{\pm \frac{\pi}{8}\im }\eta ) 
        $
         \\ 
        \hline 
         & 
         $X_{\pm }(z) = W_{\pm}'(z) $  & 
          $ Y(\eta) =   
          X_{\pm}  (e^{\pm \frac{\pi}{8}\im }\eta )  $ 
          \\ 
        \hline
    \end{tabular}
    \caption{Summary of considered variables, parameters and functions}
    \label{tab:variables_parameters_functions}
\end{table}

\noindent 
We begin with  formally considering a solution $(u,v)= \left(u(t,x,y),v(t,x,y)\right)\in \mathbb C^2$ of System \eqref{eq:linearised-Prandtl} of the form 
\begin{equation*}
    u(t,x,y) = u_k(y) e^{ \im  k x +  \sigma \sqrt{|k|} t } ,\qquad 
    v(t,x,y) =v_k(y) e^{ \im  k x +  \sigma \sqrt{|k|} t },
\end{equation*}
for general $k \in \mathbb Z\setminus\{ 0 \}$ and  $\sigma \in \mathbb C $. From the incompressibility condition, we formally look for a stream function $\phi_k : \mathbb{R}_+ \to \mathbb{C}$ with sufficient regularity (eventually in $\mathcal{C}^\infty(\mathbb{R}_+)$), such that $u_k(y) =  \phi_k'(y)$ and $v_k(y):= - \im k \phi_k(y)$ at any $y \in \mathbb R_+$. In particular, thanks to \eqref{eq:linearised-Prandtl}, $\phi_k$ satisfies the following equation and boundary conditions:
\begin{equation}\label{eq:equation-of-phik}
    \begin{cases}
    \im 
    \big( 
       - \im \sigma \sqrt{|k|} +  k  \alpha  +  k  \beta  (y-a)^2  
    \big)\phi_k'(y) -2 \im k \beta(y-a)\phi_k(y) - \phi_k'''(y) = 0,
    \qquad 
    &y \in \mathbb R_+,\\
    \phi_k(0) = \phi_k'(0) = 0.
    \end{cases}
\end{equation}
where we recall that $\alpha\in \mathbb R$, $\beta<0$ and $a\geq 0$ are the parameters of $U_\sh(y) = \alpha + \beta(y-a)^2$. 
\subsection{Rescaling around the critical point}$\,$ 

\noindent 
We next introduce the variable $z = z(y)$ and recast $\phi_k$ by a function $F_{\pm} = F_{\pm}(z)$ to reduce most of the $k$-dependence in \eqref{eq:equation-of-phik}:
\begin{equation}\label{eq:intro-def-of-z-and-F}
    z =  \sqrt[4]{ |\beta||k|}(y-a) \in  
    \big[- a \sqrt[4]{ |\beta||k|}, +\infty \big[ 
   \quad \text{and}\quad 
   F_{\pm}(z)
   := \phi_k \left(  a + \tfrac{ 1 }{ \sqrt[4]{ |\beta||k|}}z\right).
\end{equation}
The above identities are equivalent to $ y = a+ z /\sqrt[4]{ |\beta||k|}$ and $\phi_k(y) = F_{\pm}(\sqrt[4]{ |\beta||k|}(y-a))$, for any $y \in \mathbb R_+$. Thus, we can rewrite System~\eqref{eq:equation-of-phik} in terms of $F_{\pm}$ and $z$:
\begin{equation*}
    \im \left(
        -\im \sigma \sqrt{|k|} + \alpha k +  \frac{\beta k}{\sqrt{|\beta||k|}} z^2  
    \right)
    \sqrt[4]{ |\beta||k|}
    F_{\pm}'(z) - 2\im k \beta  \frac{z}{\sqrt[4]{|\beta||k|}} F_{\pm }(z) - (|\beta||k|)^\frac{3}{4}F_{\pm }'''(z) = 0,
\end{equation*}
with boundary conditions $F_{\pm}\big(- a \sqrt[4]{ |\beta||k|}\big) = 
    F_{\pm}'\left(- a \sqrt[4]{ |\beta||k|}\right) = 0 $. We divide the equation by $(|\beta||k|)^\frac{3}{4} > 0$ 
    \begin{equation*}
    \im \Big( 
       \frac{ -\im \sigma \sqrt{|k|} + \alpha k}{\sqrt{ |\beta||k|}} +  \frac{\beta k}{|\beta||k|} z^2  
    \Big)
    F_{\pm}'(z) - 2\im \frac{\beta k}{|\beta||k|} z F_{\pm }(z) - F_{\pm }'''(z) = 0.
\end{equation*}
We recall that $\beta<0 \Rightarrow \beta/|\beta|=-1$ and we set both the constant $\tau\in \mathbb C$ and the starting point $z_* \in \mathbb R$ by
\begin{equation}\label{def:tau-dependent-on-sigma-intro}
    \tau :=  -  \frac{ \im \sigma}{\sqrt{|\beta|}} \pm \alpha \sqrt{\frac{|k|}{|\beta|}} \in \mathbb C ,\qquad z_* := - a \sqrt[4]{ |\beta||k|} \in \mathbb R.
\end{equation}
The function $F_{\pm}$ is therefore a solution of the following third-order linear ordinary differential equation
\begin{equation}\label{eq:equation-of-F-intro}
     (\tau \mp  z^2)F_\pm'(z)\pm   2  z F_\pm (z) + \im  F_\pm '''(z) = 0 \qquad 
    z \in  [z_*, \infty [,
\end{equation}
where $\pm$ and $\mp$ represent ${\rm sgn}(k)$ and $-{\rm sgn}(k)$, respectively. Moreover, the boundary conditions in \eqref{eq:equation-of-phik} are reduced to $ F_\pm(z_*) =F_\pm'(z_*)  =0$. We shall remark that if $\tau =0$ and $z_* = 0$ in  \eqref{eq:equation-of-F-intro}, the function $F_{\pm}$ is up to a multiplicative constant the trivial solution $F_{\pm}(z) = z^2$. The focus of the next steps shall be on different cases.

\noindent 
The notation $F_{\pm}$ indicates that the equation \eqref{eq:equation-of-F-intro} primarily depends on $\pm = {\rm sgn}(k)$, although this is a slight abuse of notation, as both $z_*$ and $\tau$ in \eqref{def:tau-dependent-on-sigma-intro} depend on $k \in \mathbb{Z} \setminus \{0\}$. Additionally, while we focus on real $z \in [z_*, \infty[$, equation \eqref{eq:equation-of-F-intro} can also be addressed in the complex plane $z\in \mathbb C$.

\noindent 
Notably, the function $z \in [z_*, \infty[ \mapsto (\tau \mp z^2) \in \mathbb C$ satisfies the equation in \eqref{eq:equation-of-F-intro}  (though not the boundary conditions). Thus we can apply a reduction of order by introducing the ansatz   
\begin{equation}\label{eq:relation-btw-F-&-W}
    F_{\pm}(z) = (\tau \mp z^2) W_{\pm }(z),     
\end{equation}
for a suitable function $W_{\pm } = W_{\pm }(z)$ satisfying
\begin{equation}\label{eq:equation-of-W-intro}
    (\tau \mp  z^2)^2
    W_\pm'(z)
    +\im 
    \frac{\dd^3}{\dd z^3} 
    \Big(
        (\tau \mp  z^2)W_\pm(z)
    \Big)
    = 0 \qquad 
    z \in  [z_*, \infty [,
\end{equation}
The boundary conditions  are $(\tau \mp  z^2_*)^2 W_\pm(z_*)= 0$ and $(\tau \mp  z^2_*)^2 W_\pm'(z_*) \mp 2 z_*  W_\pm(z_*) = 0$, which always implies $W_{\pm}(z_*)= 0$ for $(\tau, z_*) \neq (0, 0)$  (as mentioned above, for $(\tau, z_*) =(0, 0)$, $W_{\pm}$ has to be chosen  constant). If $\tau \mp  z^2_* \neq 0$, these reduce to $W(z_*) = W'(z_*) = 0$.
\begin{remark}
For positive frequencies $k>0$, the ODE in \eqref{eq:equation-of-W-intro} coincides with the one introduced in \Cref{crit:spectral-condition-W}, with $W = W_{+}$. Moreover, recalling that $z_* = - a \sqrt[4]{ |\beta||k|}$, if $a > 0$ and $\tau \neq 0$, the function $W$ shall vanish at $z = -\infty$  in the asymptotic limit $k \to +\infty$, as stated in \Cref{crit:spectral-condition-W}. 
\end{remark}
\noindent 
The reduction of order is then complete by denoting $X_{\pm }(z):=W_{\pm }'(z)$ for $z \in [z_* , \infty[$ and formally writing the second-order ODE
\begin{equation}\label{eq:equation-of-X-intro}
    \im (\tau \mp z^2)X_{\pm}''(z) \mp 6 \im  z X_{\pm}'(z) + \left((\tau \mp z^2)^2\mp 6 \im\right)X_{\pm}(z) 
    = 0 \qquad 
    z \in  [z_*, \infty [.
\end{equation}
If $\tau\mp z_*^2 \neq 0$, the boundary condition is moreover $X_\pm(z_*) = 0$, while if $\tau\mp z_*^2 = 0$ (but $(\tau, z_*) \neq (0,0)$) any solution $X_\pm$ will yield the boundary conditions of $F_{\pm}$. While our focus is on real $z \in [z_*, \infty[$, $X_{\pm}$ can be extended meromorphically in $\mathbb{C}$ with singularities only at solutions of $\tau \mp z^2 = 0$. 

\noindent 
When $\pm = +$ (i.e.~for positive frequencies), equation \eqref{eq:equation-of-X-intro} coincides with equation \eqref{eq:X-intro} stated in the introduction.

\noindent 
Our analysis requires a further change of variable to eliminate the imaginary parts of most coefficients in both \eqref{eq:equation-of-F-intro} and  \eqref{eq:equation-of-X-intro} as well as the dependence of the equations on $\pm = {\rm sgn}(k)$.
\subsection{A further reduction through a rotation}$\,$
\label{subsection:transformation}

\noindent 
We introduce the new variable $\eta  \in \mathbb C$ together with the function $\Upsilon: \mathbb C \to \mathbb C$ by means of the relations:
\begin{equation}\label{eq:intro-def-of-eta-and-Upsilon} 
    \eta := e^{\mp \frac{\pi}{8}\im} z
    = 
    e^{\mp \frac{\pi}{8}\im}  \sqrt[4]{ |\beta||k|}(y-a)
    \in \mathbb C, \qquad 
    \Upsilon(\eta) := F_{\pm}\big(  e^{\pm \frac{\pi}{8}\im} \eta  \big)
    = \phi_k \left(  a + \tfrac{ e^{\pm \frac{\pi}{8}\im} }{ \sqrt[4]{ |\beta||k|}}\eta \right).
\end{equation}
Once again, while our main interest is in $\eta \in \mathbb{R} \cdot e^{\mp \frac{\pi}{8} \im}$, the function $F_{\pm}$ entirely extends to $\mathbb{C}$, allowing us to seek \textit{a-priori}  a function $\Upsilon$ defined in $\mathbb{C}$. By setting the parameter and initial point
\begin{equation*}
    \mu = - \im \, e^{\pm \frac{\pi}{4}\im } \tau 
    =
    - 
    \frac{e^{\pm \frac{\pi}{4}\im }}{\sqrt{|\beta|}}\Big(  \sigma \pm \im  \alpha \sqrt{|k|} \Big)
    \in \mathbb C, \qquad 
    \eta_* = e^{\mp \frac{\pi}{8}\im} z_* = 
    -a e^{\mp \frac{\pi}{8}\im}  \sqrt[4]{ |\beta||k|}\in \mathbb{R} \cdot e^{\mp \frac{\pi}{8} \im},
\end{equation*}
we obtain from System \eqref{eq:equation-of-F-intro} that $\Upsilon(\eta)= F_{\pm}(e^{\pm \frac{\pi}{8}\im} \eta)$ must satisfy the third-order linear ODE
\begin{equation*}
    \im \big( \tau \mp \big( e^{\pm \frac{\pi}{8}\im }\eta \big)^2\big) e^{\mp \frac{\pi}{8}\im }\Upsilon'(\eta) \pm 2 \im e^{\pm \frac{\pi}{8}\im }\eta\Upsilon(\eta)- e^{\mp \frac{3\pi}{8}\im }\Upsilon'''(\eta)= 0.
\end{equation*}
Multiplying the last equation by $- e^{\pm \frac{3\pi}{8}\im }$  leads  to
\begin{equation*}
    \left( \big(-\im  \tau   e^{\pm \frac{\pi}{4}\im }\big)  \pm \im e^{\pm \frac{\pi}{2}\im }\eta^2\right)  \Upsilon'(\eta) \mp 2 \im e^{\pm \frac{\pi}{2}\im }\eta\Upsilon(\eta)+\Upsilon'''(\eta)= 0,
\end{equation*}
which eventually corresponds to
\begin{equation}\label{eq:equation-of-Upsilon-intro}
    (\mu - \eta^2)\Upsilon'(\eta) +  2 \eta \Upsilon (\eta) + \Upsilon '''(\eta) = 0 \qquad 
    \eta \in  \mathbb C,
\end{equation}
as it was stated in \eqref{eq:intro-Upsilon-equation}. The accompanied boundary conditions are $\Upsilon(\eta_*) =\Upsilon'(\eta_*)  =0$. \Cref{sec:representation-via-harmonic-oscillator} is devoted to determine the general solution of \eqref{eq:equation-of-Upsilon-intro} (as well as the ones satisfying the boundary conditions). From the exact form of $\Upsilon$ we can therefore determine the exact form of $F_{\pm}$ and thus also of $\phi_k$ using \eqref{eq:intro-def-of-eta-and-Upsilon}. 

\begin{remark}\label{remark:relation-between-Upsilon-and-W}
Using \eqref{eq:relation-btw-F-&-W} and \eqref{eq:intro-def-of-eta-and-Upsilon}, the function $\Upsilon$ satisfies \eqref{eq:equation-of-Upsilon-intro} if and only if there exists $W_{\pm}(z)$ satisfying \eqref{eq:equation-of-W-intro} such that 
$\Upsilon(e^{\mp \frac{\pi i}{8}}z)=(\tau\mp z^2)W_{\pm}(z).$
By \Cref{cor:GrowthOnR}, the asymptotics of $\Upsilon$ determine those of $W_{\pm }$ as $z\to\pm\infty$.
\end{remark}

\noindent 
Finally, we introduce a last function $Y=Y(\eta)$, which plays the homologous role of $X_{\pm}$ with respect to the variable $\eta$. For any $\eta\in \mathbb C $ with $\eta^2 \neq \mu$, we set  $Y(\eta) := X_{\pm}(e^{\pm \frac{\pi}{8}\im }\eta )$ and from \eqref{eq:equation-of-X-intro}, it satisfies the equation
\begin{equation}\label{eq:Y-equation-in-the-intro}
    (\mu-\eta^2) Y''(\eta)-6\eta Y'(\eta)+\big((\mu-\eta^2)^2-6\big) Y(\eta) = 0, 
    \qquad 
    \eta \in  \mathbb C.
\end{equation}
The next section is devoted to determine explicitly all solutions of \eqref{eq:Y-equation-in-the-intro}.

\section{Proof of \Cref{thm:explicit-form-of-X}}\label{sec:explicit-solutions-of-Y-and-X}
\noindent 
The following section is devoted to the proof of  \Cref{thm:explicit-form-of-X} and to determine the general solution $X=X(z)$ of Equation \eqref{eq:X-intro}. We rather address the equivalent form in $Y= Y(\eta)= X\left(e^{\frac{\pi  }{8} \im}\eta\right)$, that satisfies equation \eqref{eq:Y-equation-in-the-intro} with $\mu = \tau e^{-\frac{\pi }{4}\im}$.
\begin{prop}\label{prop:special-solutions-of-Y-in-power-series}
    Let $\mu\in \mathbb C$ be an arbitrary complex number, and define $\mathcal A_\mu = \{ \eta \in \mathbb C~|~\eta^2 = \mu \}$.   
    All solutions $Y:\mathbb C\setminus \mathcal A_\mu \to \mathbb C$ of the complex ordinary differential equation
    \begin{equation}\label{eq:Y-prop}
        (\mu-\eta^2) Y''(\eta)-6\eta Y'(\eta)+\big((\mu-\eta^2)^2-6\big) Y(\eta) = 0,
        \qquad \eta \in \mathbb C \setminus \mathcal A_\mu,
    \end{equation}
    can explicitly be  written as
    \begin{equation*}
        Y(\eta) = c_1 Y_{\mu,1}(\eta)+ c_1 Y_{\mu,2}(\eta), \qquad \eta \in \mathbb C \setminus \mathcal A_\mu,
    \end{equation*}
    where $c_1$ and $c_2$ are arbitrary complex numbers, while the functions $Y_{\mu,1}$ and $Y_{\mu,2}$ are given by 
    \begin{equation}\label{def:power-series-in-prop-R}
    \begin{aligned}
        Y_{\mu,1}(\eta) 
        &:= 
        \dfrac{\exp\big( -\frac{\eta^2}{2}\Big)}{(\mu-\eta^2)^2}R_{\mu, 1}(\eta)\quad \text{with }\quad 
        R_{\mu, 1}(\eta) := \sum_{n = 0}^\infty \frac{\mu-  2n}{n!}
        \frac{\left( -\frac{\mu+1}{4}\right)_n}{\left( \frac{1}{2}\right)_n}\eta^{2n},
        \\
        Y_{\mu,2}(\eta) 
        &:= 
        \dfrac{\exp\big( -\frac{\eta^2}{2}\Big)}{(\mu-\eta^2)^2} R_{\mu, 2}(\eta)\quad \text{with }\quad 
        R_{\mu, 2}(\eta) :=
        \eta+ \frac{1}{4} \sum_{n = 1}^\infty \frac{2n+1-\mu}{n!}
        \frac{\left( \frac{5-\mu}{4}\right)_{n-1}}{\left( \frac{3}{2}\right)_n}\eta^{2n+1}.
    \end{aligned}
    \end{equation}
\end{prop}
\begin{remark}
    The power series $R_{\mu, 1}$ and $R_{\mu, 2}$, as defined in \eqref{def:power-series-in-prop-R}, are even and odd, respectively. Both have an infinite radius of convergence, which implies that they define entire functions in $\mathbb{C}$. The functions $Y_{\mu,1}$ and $Y_{\mu,2}$ are therefore meromorphic across $\mathbb{C}$, exhibiting singularities where $\eta^2 = \mu$ and having at each singular point residue equal to zero. Consequently, both functions possess primitive functions that are meromorphic over the same domain.
\end{remark}
\noindent 
Before proceeding with the proof of \Cref{prop:special-solutions-of-Y-in-power-series}, we present a corollary (whose proof is at the end of this section), that expresses the solutions $Y_{\mu, 1}$ and $Y_{\mu, 2}$ in \eqref{def:power-series-in-prop-R} in a compact form depending on the Kummer's hypergeometric function $\mathcal M$.
\begin{cor}\label{cor:special-solutions-of-Y}
    Under the conditions of \Cref{prop:special-solutions-of-Y-in-power-series}, the functions $Y_{\mu, 1}$ and $Y_{\mu, 2}$ defined in \eqref{def:power-series-in-prop-R} satisfy 
    \begin{equation*}
    \begin{aligned}
        Y_{\mu,1}(\eta) 
        &:= 
        \dfrac{\exp\big( -\frac{\eta^2}{2}\Big)}{(\mu-\eta^2)^2}
        \bigg[ 
            \mu\, 
            \mathcal M\Big( -\frac{\mu+1}{4},\frac{1}{2},\eta^2\Big)
            +
            (\mu+1)
            \eta^2\,
            \mathcal M\Big( \frac{3-\mu}{4},\frac{3}{2},\eta^2\Big)
        \bigg],
        \\
        Y_{\mu,2}(\eta) 
        &:= 
        \eta
        \dfrac{\exp\big( -\frac{\eta^2}{2}\Big)}{(\mu-\eta^2)^2}
        \bigg[  
            \mathcal M\Big( \frac{1-\mu}{4},\frac{3}{2},\eta^2\Big)
            +
            \frac{\eta^2}{3}
            \mathcal M\Big( \frac{5-\mu}{4},\frac{5}{2},\eta^2\Big)
        \bigg],
    \end{aligned}
    \end{equation*}
    where $\mathcal M$ is the Kummer's function introduced in \eqref{def:kummer's-function}.
\end{cor}
\noindent 
 \Cref{thm:explicit-form-of-X} follows from \Cref{cor:special-solutions-of-Y} by setting $\mu = \tau e^{-\frac{\pi \im}{4}}$ and $\eta = e^{-\frac{\pi \im}{8}}z$.
\begin{proof}[Proof of \Cref{prop:special-solutions-of-Y-in-power-series}]
    We  carry out the proof by  three major steps:
    \begin{enumerate}[(i)]
    \item We reformulate the ordinary differential equation in \eqref{eq:Y-prop} in terms of a new function $R:\mathbb C\setminus \mathcal{A}_\mu \to \mathbb C$, setting
    \begin{equation}\label{eq:def-of-R}
    \begin{aligned}
        Y(\eta)
        &=  
        \dfrac{e^{ -\frac{\eta^2}{2}}}{(\mu-\eta^2)^2}R(\eta).
    \end{aligned}
    \end{equation}
    We show that this transformation recasts \Cref{eq:Y-prop} into the following ODE for $R$:
    \begin{equation}\label{eq:R-prop}
        (\mu-\eta^2)R''(\eta)-2\eta (\mu-1-\eta^2 ) R'(\eta) + (\mu+1)(\mu-2-\eta^2)R(\eta) = 0\qquad \eta \in \mathbb C\setminus \mathcal{A}_\mu.
    \end{equation}
    This is stated in \Cref{appx-lemma:ODEY->ODER}.
    \item  We seek two independent non-trivial solutions $R$ of \eqref{eq:R-prop} making use of a power series approach. To this end, we  consider two general sequences $(a_n)_{n \in \mathbb{N}_0} \subseteq \mathbb{C}$ and $(b_n)_{n \in \mathbb{N}_0} \subseteq \mathbb{C}$, assuming momentary that both
    \begin{equation}\label{def:R12-power-series-in-the-proof}
        R_{\mu,1}(\eta) :=  \sum_{n = 0}^\infty a_n \eta^{2n},\qquad 
        R_{\mu,2}(\eta) := \sum_{n = 0}^\infty b_n \eta^{2n+1}.
    \end{equation}
    have positive radii of convergence.  We then identify the recursive conditions on these sequences ensuring that $R_{\mu,1}$ and $R_{\mu,2}$ are solutions of \eqref{eq:R-prop} within their domain of definition. These are eventually:
     \begin{equation}\label{eq:cond-an-prop-R}
        \begin{cases}
         2\mu \,a_1 + (\mu^2-\mu-2) a_0  = 0, \quad &\\
          2\mu  (2n^2+3n+1) a_{n+1}  -
        \big(  4n^2+(4\mu-6)n-  \mu^2+\mu+2 \big)  a_n 
        +
       \big(  4 n - 5 -\mu \big) a_{n-1} =  0 & \; \forall n \in \mathbb N,
        \end{cases}
    \end{equation}
    and
    \begin{equation}\label{eq:cond-bn-prop-R}
        \begin{cases}
         6 \mu b_1 +(\mu^2-3\mu) b_0 = 0, \quad &\\
         2\mu (2n^2+5n+3) b_{n+1} - \big(  4n^2+(4\mu-2)n- \mu^2 + 3\mu  \big)b_n
            +
            (4n -3 -\mu)b_{n-1} =  0 & \quad\quad \forall n \in \mathbb N.
        \end{cases}
    \end{equation}
    This is stated in \Cref{lemma:transforming-R-eq-into-recursive-sequences}.
    \item We demonstrate that both \eqref{eq:cond-an-prop-R} and \eqref{eq:cond-bn-prop-R} are met by the sequences written in \eqref{def:power-series-in-prop-R}, namely
    \begin{equation}\label{def:an-bn-prop-R}
        a_n = 
        \frac{\mu-  2n}{n!}
        \frac{\left( -\frac{\mu+1}{4}\right)_n}{\left( \frac{1}{2}\right)_n} \quad  \forall n \in \mathbb N_0 \quad \text{and}\quad 
        b_n =
        \begin{cases}
        1 & \text{if }n =0,
        \\
        \frac{1}{4} 
        \frac{2n+1-\mu}{n!}
        \frac{\left( \frac{5-\mu}{4}\right)_{n-1}}{\left( \frac{3}{2}\right)_n}\quad 
        &\text{if }n \in \mathbb N.
        \end{cases}
    \end{equation}
    The conclusion of the proof follows from the fact that the series $R_{1,\mu}$ and $R_{2,\mu}$, generated by these sequences, are entire with infinite radii of convergence, making them solutions of \eqref{eq:R-prop}, in fact, in all $\mathbb{C}$.
    \end{enumerate}
    Given the technicality of Part {\rm (i)} and Part {\rm (ii)}, we postpone the  related proofs to \Cref{appx-lemma:ODEY->ODER} and \Cref{lemma:transforming-R-eq-into-recursive-sequences} below, respectively. We focus the next steps on proving Part {\rm (iii)}, assuming momentarily that {\rm (i)} and {\rm (ii)} hold true. 
    The sequence $(a_n)_{n\in \mathbb N_0}$ in \eqref{def:an-bn-prop-R} satisfies the first identity of \eqref{eq:cond-an-prop-R}, since
    \begin{align*}
    2\mu  a_1 + &(\mu^2-\mu-2)a_0
    =  2\mu (\mu-2)\frac{-\frac{\mu+1}{4}}{\frac{1}{2}} + (\mu^2-\mu-2) \mu 
    = - \mu (\mu-2) (\mu+1)+ (\mu^2-\mu-2)\mu = 0.   
    \end{align*}
    Similarly, the sequence $(b_n)_{n\in \mathbb N_0}$ in \eqref{def:an-bn-prop-R} satisfies the first identity of \eqref{eq:cond-bn-prop-R}, since
    \begin{align*}
    6 \mu b_1 +(\mu^2-3\mu) b_0 
    =  6 \mu \frac{1}{4} (3-\mu)\frac{1}{\frac{3}{2}}+ (\mu^2-3\mu)
    =  \mu (3-\mu) + \mu(\mu-3) = 0.
    \end{align*}
    It remains therefore to show that the recursive formula in \eqref{eq:cond-an-prop-R} and \eqref{eq:cond-bn-prop-R} are satisfied. We note that the first term in the second equation of \eqref{eq:cond-an-prop-R} can be written as
    \begin{equation}\label{eq:prop-R-showing-an-satisfies-relations-1}
    \begin{aligned}
        2\mu  (2n^2+3n+1) a_{n+1} 
        &=
        2\mu (n+1)(2n+1) \frac{\mu-  2(n+1)}{(n+1)!}
        \frac{4n-1-\mu }{2(2n+1)}
        \frac{\left( -\frac{\mu+1}{4}\right)_{n}}{\left( \frac{1}{2}\right)_{n}}\\
        &=
        \mu 
        (-  2n+\mu-2)(4n-1-\mu )
        \frac{1}{n!}
        \frac{\left( -\frac{\mu+1}{4}\right)_{n}}{\left( \frac{1}{2}\right)_{n}}\\
        &=
        \Big(-8\mu\, n^2+(6\mu^2-6\mu)n-\mu^3+\mu^2+2\mu \Big)
        \frac{1}{n!}
        \frac{\left( -\frac{\mu+1}{4}\right)_{n}}{\left( \frac{1}{2}\right)_{n}},
    \end{aligned}
    \end{equation}
    for any $n \in \mathbb N$. Additionally, for the second term of the recursive formula in \eqref{eq:cond-an-prop-R}, we have that
    \begin{equation}\label{eq:prop-R-showing-an-satisfies-relations-2}
    \begin{aligned}
        -\big(  4n^2+(4\mu-6)n-  \mu^2+\mu+2  \big)  a_n 
        = 
        -\big(  4n^2+(4\mu-6)n- \mu^2+\mu+2 \big)(\mu-2n)
        \frac{1}{n!}
        \frac{\left( -\frac{\mu+1}{4}\right)_{n}}{\left( \frac{1}{2}\right)_{n}}\\
        = 
        \Big( 8n^3+(4\mu-12)n^2+(4+8\mu-6\mu^2)n+\mu^3-\mu^2-2\mu\Big)
        \frac{1}{n!}
        \frac{\left( -\frac{\mu+1}{4}\right)_{n}}{\left( \frac{1}{2}\right)_{n}}.
    \end{aligned}    
    \end{equation}
    Finally, the last term in the recursive formula of \eqref{eq:cond-an-prop-R} can be expressed as 
    \begin{equation}\label{eq:prop-R-showing-an-satisfies-relations-3}
    \begin{aligned}
        (  4 n &- 5 -\mu ) a_{n-1} 
        = 
        (2n-1)\frac{  4 (n-1) - (\mu+1) }{2n-1} a_{n-1}
        =
        (2n-1)\frac{4}{2} \frac{-\frac{\mu+1}{4}+n-1}{\frac{1}{2}+n-1}a_{n-1}\\
        &=
        (2n-1)2 
         \frac{-\frac{\mu+1}{4}+n-1}{\frac{1}{2}+n-1}
        \frac{\mu- 2n+2}{(n-1)!}
        \frac{\left( -\frac{\mu+1}{4}\right)_{n-1}}{\left( \frac{1}{2}\right)_{n-1}}
        =
        2n(2n-1)(\mu- 2n+2)
        \frac{1}{n!}
        \frac{\left( -\frac{\mu+1}{4}\right)_{n}}{\left( \frac{1}{2}\right)_{n}}
        \\
        &=
        \Big( -8n^3+(4\mu+12)n^2-(4+2\mu)n \Big) 
        \frac{1}{n!}
        \frac{\left( -\frac{\mu+1}{4}\right)_n}{\left( \frac{1}{2}\right)_n}.
    \end{aligned}    
    \end{equation}
    The total sum of the right-hand-sides in the identities \eqref{eq:prop-R-showing-an-satisfies-relations-1}, \eqref{eq:prop-R-showing-an-satisfies-relations-2} and \eqref{eq:prop-R-showing-an-satisfies-relations-3} is identically null. This implies that the sequence $(a_n)_{n \in \mathbb N_0}$ satisfies the second equation in \eqref{eq:cond-an-prop-R}. 

    \noindent
    A similar argument holds true also for the sequence $(b_n)_{n\in \mathbb N_0}$ in \eqref{def:an-bn-prop-R} and the second equation in \eqref{eq:cond-bn-prop-R}. We indeed remark that for $n = 1$
    \begin{align*}
        2\mu 
        &(2n^2+5n+3) b_{n+1} - \big(  4n^2+(4\mu-2)n- \mu^2 + 3\mu  \big)b_n
            +
            (4n -3 -\mu)b_{n-1} 
        \\    
        &=
        20\,\mu\, b_2 - (2+7\mu -\mu^2) b_1     +(1-\mu) b_0 
        =
        20 \mu \frac{(5-\mu)^2}{120}-
        (2+7\mu -\mu^2)\frac{3-\mu}{6}+(1-\mu) \\    
        &=
       \frac{\mu^3-10\mu^2 +25\mu }{6}-
        \frac{6+19\mu -10 \mu^2+\mu^3}{6}+1-\mu
        = 0.
    \end{align*}
    Moreover, with $n\geq 2$, the first term 
    $2\mu (2n^2+5n+3) b_{n+1}$ in the second equation of \eqref{eq:cond-bn-prop-R} is
    \begin{equation}\label{eq:prop-R-showing-bn-satisfies-relations-1}
    \begin{aligned}
        2\mu  (2n^2+5n+3) b_{n+1} 
        &=
        2\mu  (2n+3)(n+1) 
        \frac{1}{4}
        \frac{2n+3-\mu}{(n+1)!}
        \frac{\left( \frac{5-\mu}{4}\right)_{n}}{\left( \frac{3}{2}\right)_{n+1}}\\
        &=
        \frac{\mu}{2} 
        (2n+3)(2n+3-\mu) 
        \frac{1}{n!}
        \frac{
        \left(\frac{5-\mu}{4}+n-1 \right)
        \left( \frac{5-\mu}{4}\right)_{n-1}}{\left(\frac{3}{2}+n \right)\left( \frac{3}{2}\right)_{n}}\\
        &=
        \frac{\mu}{2} 
        (2n+3)(2n+3-\mu) 
        \frac{1}{n!}
        \frac{1}{2}
        \frac{4n+1-\mu}{2n+3}
        \frac{
        \left( \frac{5-\mu}{4}\right)_{n-1}}{\left( \frac{3}{2}\right)_{n}}\\
        &=
        \frac{\mu}{4} 
        (2n+3-\mu)(4n+1-\mu)
        \frac{1}{n!}
        \frac{
        \left( \frac{5-\mu}{4}\right)_{n-1}}{\left( \frac{3}{2}\right)_{n}}\\
        &=
        \frac{1}{4} 
        \Big(8\mu\, n^2+(14\mu-6\mu^2)n+\mu^3-4\mu^2+3\mu \Big)
        \frac{1}{n!}
        \frac{
        \left( \frac{5-\mu}{4}\right)_{n-1}}{\left( \frac{3}{2}\right)_{n}},
    \end{aligned}
    \end{equation}
    while the second term in the recursive formula of \eqref{eq:cond-bn-prop-R} is
    \begin{equation}\label{eq:prop-R-showing-bn-satisfies-relations-2}
    \begin{aligned}
        - \big(  4n^2+(4\mu-2)n- \mu^2 + 3\mu  \big)  b_n 
        = 
        -\big(  4n^2+(4\mu-2)n- \mu^2+3\mu \big)
        \frac{1}{4}\frac{2n+1-\mu}{n!}
        \frac{\left(\frac{5-\mu}{4}\right)_{n-1}}{\left( \frac{3}{2}\right)_{n}}\\
        =         
        \frac{1}{4}
        \Big( 
           - 8n^3- 4\mu n^2 + (6\mu^2 -12 \mu +2 )n -\mu^3 +4 \mu^2 -3 \mu  \Big)
        \frac{1}{n!}
        \frac{\left( \frac{5-\mu}{4}\right)_{n-1}}{\left( \frac{3}{2}\right)_{n}}.
    \end{aligned}    
    \end{equation}    
    Finally, the last term in the recursive formula of \eqref{eq:cond-bn-prop-R} can be expressed as 
    \begin{equation}\label{eq:prop-R-showing-bn-satisfies-relations-3}
    \begin{aligned}
        (  4 n& - 3 -\mu ) b_{n-1} 
        = 
        (2n+1)\frac{4 (n-2) + (5 -\mu) }{2n+1} b_{n-1} 
        =
        (2n+1)\frac{4}{2} \frac{\frac{5-\mu}{4}+n-2}{\frac{3}{2}+n-1}b_{n-1}\\
        &=
        (2n+1)
        2\frac{\frac{5-\mu}{4}+n-2}{\frac{3}{2}+n-1}
        \frac{1}{4}
        \frac{2n-1-\mu}{(n-1)!}        
        \frac{\left( \frac{5-\mu}{4}\right)_{n-2}}{\left( \frac{3}{2}\right)_{n-1}}
        =
        \frac{1}{2}(2n+1)(2n-1-\mu)
        \frac{n}{n!}        
        \frac{\left( \frac{5-\mu}{4}\right)_{n-1}}{\left( \frac{3}{2}\right)_{n}}
        \\
        &=
        \frac{1}{4}
        \Big( 8n^3-4 \mu n^2+(-2-2\mu)n \Big) 
        \frac{1}{n!}
        \frac{\left( \frac{5-\mu}{4}\right)_n}{\left( \frac{3}{2}\right)_n}.
    \end{aligned}    
    \end{equation}
    The total sum of the right-hand-sides in the identities \eqref{eq:prop-R-showing-bn-satisfies-relations-1}, \eqref{eq:prop-R-showing-bn-satisfies-relations-2} and \eqref{eq:prop-R-showing-bn-satisfies-relations-3} is identically null. This implies that the sequence $(b_n)_{n \in \mathbb N_0}$ satisfies the recursive equation in \eqref{eq:cond-bn-prop-R}. This concludes the proof of \Cref{prop:special-solutions-of-Y-in-power-series}.
\end{proof}
\noindent The following lemma addresses Part {\rm (i)} of \Cref{prop:special-solutions-of-Y-in-power-series}, reformulating the ODE \eqref{eq:Y-prop} for $Y$ as the equivalent ODE \eqref{eq:R-prop} in terms of the function $R$ in \eqref{eq:def-of-R}.
\begin{lemma}\label{appx-lemma:ODEY->ODER}
     Let $\mu\in \mathbb C$ be arbitrary and let $\mathcal A_\mu = \{ \eta \in \mathbb C~|~ \eta^2 = \mu \}$. A function $Y:\mathbb C\setminus A_\mu \to \mathbb C$ is a solution of 
    \begin{equation}\label{eq:Y-appx}
        (\mu-\eta^2) Y''(\eta)-6\eta Y'(\eta)+\big((\mu-\eta^2)^2-6\big) Y(\eta) = 0,
        \qquad \eta \in \mathbb C \setminus \mathcal A_\mu,
    \end{equation}
    if and only if there exists a function $R: \mathbb C \setminus A_\mu \to \mathbb C$ that satisfies both \eqref{eq:def-of-R} and the ODE
    \begin{equation}\label{eq:R-appx}
        (\mu-\eta^2)R''(\eta)-2\eta (\mu-1-\eta^2 ) R'(\eta) + (\mu+1)(\mu-2-\eta^2)R(\eta) = 0
        \qquad \eta \in \mathbb C \setminus \mathcal A_\mu.
    \end{equation}
\end{lemma}
\begin{proof} The proof follows a straightforward approach: the equation \eqref{eq:Y-appx} is reformulated so that every term involving $Y$ aligns with the form given by $R$ in \eqref{eq:def-of-R}. We first observe that \eqref{eq:Y-appx} can be written as
    \begin{equation*}
         \frac{1}{(\mu-\eta^2)^2}\dfrac{d}{d\eta }\Big( (\mu-\eta^2)^3 Y'(\eta)\Big) +
         \big((\mu-\eta^2)^2-6\big) Y(\eta) = 0,\qquad \eta \in \mathbb C \setminus A_\mu.
    \end{equation*}
    Using the product rule, we hence collect outside the derivative of $Y'(\eta)$:
    \begin{equation*}
         \frac{1}{(\mu-\eta^2)^2}\dfrac{d^2}{d\eta^2 }\Big( (\mu-\eta^2)^3 Y(\eta)\Big) 
         +
         \frac{1}{(\mu-\eta^2)^2}\dfrac{d}{d\eta }
         \Big( 6\eta (\mu-\eta^2)^2 Y(\eta)\Big) 
         +
         \big((\mu-\eta^2)^2-6\big) Y(\eta) = 0.
    \end{equation*}
    Hence, developing the derivative in the second term while retaining the product $(\mu-\eta^2)^2 Y(\eta)$ together, we get
    \begin{equation}\label{eq:deriving-R-equation-part1}
         \frac{1}{(\mu-\eta^2)^2}\dfrac{d^2}{d\eta^2 }\Big( 
         (\mu-\eta^2)^3
         Y(\eta)\Big) 
         +
         \frac{6\eta }{(\mu-\eta^2)^2}\dfrac{d}{d\eta }
         \Big(  (\mu-\eta^2)^2 Y(\eta)\Big) 
         +
         (\mu-\eta^2)^2 Y(\eta) = 0.
    \end{equation}
    Next, we handle the first term with the second derivatives and we rewrite it to highlight the dependence on the function $(\mu - \eta^2)^2 Y(\eta)$:
    \begin{align*}
         \dfrac{d^2}{d\eta^2 }
         \Big( 
            (\mu-\eta^2)^3
            Y(\eta)
        \Big)
        &=
         \dfrac{d^2}{d\eta^2 }
         \Big( 
            (\mu-\eta^2)
            (\mu-\eta^2)^2
            Y(\eta)
        \Big)
        \\
        &=         
        (\mu-\eta^2)
        \dfrac{d^2}{d\eta^2 }\Big( (\mu-\eta^2)^2 Y(\eta)\Big) 
         -4\eta 
         \dfrac{d}{d\eta }\Big( (\mu-\eta^2)^2 Y(\eta)\Big) 
         -
         2
         (\mu-\eta^2)^2 Y(\eta).
    \end{align*}
    Plugging the last relation into \eqref{eq:deriving-R-equation-part1} yields therefore to the following identity:
    \begin{align*}
         \frac{1}{ \mu-\eta^2}
         \dfrac{d^2}{d\eta^2 }\Big( (\mu-\eta^2)^2 Y(\eta)\Big) 
        +
        \frac{2\eta }{(\mu-\eta^2)^2}\dfrac{d}{d\eta }
         \Big(  (\mu-\eta^2)^2 Y(\eta)\Big) 
         +
         \big( (\mu-\eta^2)^2 -2 \big)Y(\eta) = 0.
    \end{align*}
    We next set $R(\eta) = e^{\frac{\eta^2}{2}}(\mu- \eta^2) Y(\eta)$ and multiply the resulting equation by $(\mu-\eta^2)^2$. We thus obtain
    \begin{equation}\label{eq:deriving-R-equation-part2}
        (\mu-\eta^2) 
        \dfrac{d^2}{d\eta^2 }\Big( e^{-\frac{\eta^2}{2}} R(\eta)\Big) 
        +
        2\eta 
        \dfrac{d}{d\eta }\Big( e^{-\frac{\eta^2}{2}} R(\eta)\Big) 
        +
        \big( 
            (\mu-\eta^2)^2 -2
        \big)    
        e^{-\frac{\eta^2}{2}} 
        R(\eta) = 0.
    \end{equation}
    Hence, thanks to the identities
    \begin{equation*}
        \dfrac{d}{d\eta }\Big( e^{-\frac{\eta^2}{2}} R(\eta)\Big) 
        =
        e^{-\frac{\eta^2}{2}}  
       \big(
        R'(\eta )- \eta  R(\eta ) 
       \big),
       \quad 
       \dfrac{d^2}{d\eta^2}\Big( e^{-\frac{\eta^2}{2}} R(\eta)\Big) 
        =
       e^{-\frac{\eta^2}{2}}  
       \Big(
         R''(\eta)
         -2\eta R'(\eta)+
        (\eta^2-1)R(\eta)
       \Big),
    \end{equation*}
    we gather that the last equation in \eqref{eq:deriving-R-equation-part2} corresponds to
    \begin{equation*}
       (\mu-\eta^2)  R''(\eta) 
       +
       \big((\mu-\eta^2)(-2\eta) + 2\eta \big)R'(\eta) 
       +
       \Big((\mu-\eta^2)(\eta^2-1)+2\eta (-\eta) + (\mu-\eta^2)^2-2\Big)R(\eta) = 0.
    \end{equation*}
    Finally, since $(\mu-\eta^2)(\eta^2-1) + (\mu-\eta^2)^2 = 
    (\mu-\eta^2)(\mu-1)= (\mu-\eta^2)(\mu+1)- 2\mu +2\eta^2$, we obtain
    \begin{equation*}
       (\mu-\eta^2)  R''(\eta) -2\eta (\mu-1- \eta^2)R'(\eta) 
       +
       \Big( 
        (\mu-\eta^2)
        (\mu+1)
        -2\mu
        -2
        \Big)
        R(\eta) = 0,
    \end{equation*}
    which, together with $(\mu-\eta^2)
        (\mu-\eta^2)(\mu+1)
        -2\mu
        -2
        = (\mu+1)(\mu-2-\eta^2)$, correspond to the ODE in \eqref{eq:R-appx}.
\end{proof}
\noindent 
To complete the proof of \Cref{prop:special-solutions-of-Y-in-power-series}, it remains to establish Part {\rm (ii)}, concerning the recursive conditions required for the power series $R_{\mu, 1}$ and $R_{\mu, 2}$ in \eqref{def:R12-power-series-in-the-proof} to satisfy (at least locally) the ODE in \eqref{eq:R-prop}.
\begin{lemma}\label{lemma:transforming-R-eq-into-recursive-sequences}
    Let $R_{\mu, 1}(\eta) = \sum_{n=0}^\infty a_n \eta^{2n}$ and $R_{\mu,2}(\eta) = \sum_{n=0}^\infty b_n \eta^{2n+1}$ be complex power series with nonzero radii of convergence. Then $R_{\mu, 1}$ and $R_{\mu, 2}$ are independent local solutions of the ODE \eqref{eq:R-prop} if and only if the sequences $(a_n)_{n \in \mathbb{N}_0}$ and $(b_n)_{n \in \mathbb{N}_0}$ satisfy the recursive conditions in \eqref{eq:cond-an-prop-R} and \eqref{eq:cond-bn-prop-R}.
\end{lemma}
\begin{proof}
   We begin by expanding each component of the ODE \eqref{eq:R-prop}. When applied to the power series $R = R_{\mu,1}$, the first term in \eqref{eq:R-prop} corresponds to
    \begin{equation}\label{eq:Prop-Y-R''-a} 
    \begin{aligned}
        (\mu-\eta^2)R_{\mu,1}''(\eta)
        &=
         (\mu-\eta^2)  \sum_{n = 1}^\infty 2n (2n-1) a_n \eta^{2(n-1)}
        =
        \mu \sum_{k = 0}^\infty 2(k+1)(2k+1) a_{k+1}\eta^{2k}+
        \\
        &-
        \sum_{n = 1}^\infty 2n (2n-1) a_n \eta^{2n}
        =
        2\mu\, a_1  +
        \sum_{n = 1}^\infty \big( 2\mu  (n+1)(2n+1) a_{n+1}  -   2n (2n-1) a_n 
        \big)\eta^{2n}.
    \end{aligned}
    \end{equation} 
    Next, we expand the second term $-2\eta (\mu-1-\eta^2 ) R'(\eta)$ in \eqref{eq:R-prop}. when replacing $R$ with $R_{\mu,1}$:
    \begin{equation}\label{eq:Prop-Y-R'-a}
    \begin{aligned}
        -2\eta (&\mu-1-\eta^2 ) R_{\mu,1}'(\eta)
        =
        -2\eta (\mu-1-\eta^2 ) \sum_{n = 1}^\infty 2n\,  a_n \eta^{2n-1}\\
        &=
        -2(\mu-1) \sum_{n = 1}^\infty 2n\, a_{n}\eta^{2n}
        +
        2 
        \sum_{ k= 2 }^\infty 2(k-1) \, a_{k-1} \eta^{2k}=
       \sum_{ n= 1 }^\infty \big( - 
        4(\mu-1) n\, a_n +4 \underbrace{(n-1)a_{n-1}}_{=0 \text{ if }n = 1}  \big)\eta^{2n}.
    \end{aligned}
    \end{equation} 
    Finally, the last term $(\mu+1)(\mu-2-\eta^2)R(\eta) $ in \Cref{eq:R-prop} with $R = R_{\mu,1}$ can be expressed as
    \begin{equation}\label{eq:Prop-Y-R-a}
    \begin{aligned}
        (\mu+1)&(\mu-2-\eta^2)R_{\mu,1}(\eta)
        =
        (\mu+1)(\mu-2-\eta^2) \sum_{n = 0}^\infty   a_n \eta^{2n}=
        (\mu+1)(\mu-2) \sum_{n = 0}^\infty a_{n}\eta^{2n} 
        +
        \\
        &-
        (\mu+1)
        \sum_{k = 1}^\infty a_{k-1} \eta^{2k} 
        =
        (\mu+1)(\mu-2) a_0
        +
         \sum_{n = 1}^\infty 
         \big(
           (\mu+1)(\mu-2) a_{n}
            -
            (\mu+1)
            a_{n-1} 
        \big)
        \eta^{2n}.
    \end{aligned}
    \end{equation}
    Summing the expressions in \eqref{eq:Prop-Y-R''-a}, \eqref{eq:Prop-Y-R'-a} and \eqref{eq:Prop-Y-R-a}, and observing that the following identities hold true
    \begin{equation*}
    \begin{aligned}
         -   2n (2n-1)a_n- 4 (\mu-1) n a_n + (\mu+1)(\mu-2)a_n &= 
        \big( - 4n^2+(-4\mu+6)n +\mu^2-\mu -2 \big)a_n,\\
         4 (n-1)a_{n-1}  - (\mu+1)a_{n-1} &= (4n -5 -\mu)a_{n-1},
    \end{aligned}
    \end{equation*}
    we obtain that the power series $R = R_{\mu, 1}$ is a local solution of \Cref{eq:R-prop} if and only if the following power series is identically null:
    \begin{equation*}
        2\mu a_1 + (\mu+1)(\mu-2)a_0 + 
        \sum_{n = 1}^\infty 
        \Big\{ 
            2\mu (n+1)(2n+1) a_{n+1}- \big[  4n^2+(4\mu-6)n- \mu^2 +\mu + 2 \big]a_n
            +
            (4n-5-\mu)a_{n-1}
        \Big\}\eta^{2n} = 0.
    \end{equation*}
    Since $(\mu+1)(\mu-2) = \mu^2 - \mu - 2$ and $2\mu(n+1)(2n+1) = 2\mu(2n^2 + 3n + 1)$, this holds if and only if the sequence $(a_n)_{n \in \mathbb{N}}$ satisfies the recursive relations specified in \eqref{eq:cond-an-prop-R}.

    \noindent 
    We consider next the power series $R_{\mu,2}$. Replacing $R =R_{\mu,2} $ in the first term $(\mu-\eta^2)R''(\eta)$ of \eqref{eq:R-prop}, we obtain
    \begin{equation}\label{eq:Prop-Y-R''-b}
    \begin{aligned}
        (\mu-\eta^2)R_{\mu,2}''(\eta)
        &=
         (\mu-\eta^2)  \sum_{n = 1}^\infty (2n+1)2n \,b_n \eta^{2n-1}=
        \mu \sum_{k = 0}^\infty (2k+3)(2k+2) b_{k+1}\eta^{2k+1}
        +
        \\
        &-
        \sum_{n = 1}^\infty (2n+1)2n \,b_n  \eta^{2n+1}
        =
        6\mu\, b_1 \eta  +
        \sum_{n = 1}^\infty \big( 2\mu (2n+3)(n+1) b_{n+1}  -  2n (2n+1) b_n 
        \big)\eta^{2n+1}.
    \end{aligned}
    \end{equation}   
    Analogously, setting $R = R_{2,\mu}$ into the second term $-2\eta (\mu - 1 - \eta^2) R'(\eta)$ of \eqref{eq:R-prop}, we get
    \begin{equation}\label{eq:Prop-Y-R'-b}
    \begin{aligned}
        -2\eta &(\mu-1-\eta^2 ) R_{\mu,2}'(\eta)
        =
        -2\eta (\mu-1-\eta^2 ) \sum_{n = 0}^\infty (2n+1)  b_n \eta^{2n}=
        -2(\mu-1) \sum_{n = 0}^\infty  (2n+1)  b_{n}\eta^{2n+1}
        +\\
        &+
        2 
        \sum_{ k= 1 }^\infty (2k-1) \, b_{k-1} \eta^{2k+1}=
        -2(\mu-1)b_0 \eta 
        +\sum_{ n= 1 }^\infty \Big(-2(\mu-1)(2 n+1)b_n +2(2n-1)b_{n-1}  \Big)\eta^{2n+1}.
    \end{aligned}
    \end{equation} 
    Finally, the last term $(\mu+1)(\mu-2-\eta^2)R(\eta) $ in \Cref{eq:R-prop} with $R = R_{\mu,2}$ corresponds to 
    \begin{equation}\label{eq:Prop-Y-R-b}
    \begin{aligned}
        (\mu+1)&(\mu-2-\eta^2)R_{\mu,2}(\eta)
        =
        (\mu+1)(\mu-2-\eta^2) \sum_{n = 0}^\infty   b_n \eta^{2n+1}=
        (\mu+1)(\mu-2) \sum_{n = 0}^\infty b_{n}\eta^{2n+1}
        +
        \\
        &-
        (\mu+1)
        \sum_{k = 1}^\infty b_{k-1} \eta^{2k+1} =
        ( \mu^2-\mu-2) b_0\eta 
        +
         \sum_{n = 1}^\infty 
         \big(
           (\mu^2-\mu-2) b_{n}
            -
            (\mu+1)
            b_{n-1} 
        \big)
        \eta^{2n+1}.
    \end{aligned}
    \end{equation}
    Summing the expressions in \eqref{eq:Prop-Y-R''-b}, \eqref{eq:Prop-Y-R'-b} and \eqref{eq:Prop-Y-R-b}, and observing that the following identities hold true
    \begin{equation*}
    \begin{aligned}
         -   2n(2n+1)  b_n -2(\mu-1)(2 n+1) b_n+ ( \mu^2-\mu-2)b_n&= 
        \big(  -4n^2+(-4\mu+2)n + \mu^2 - 3\mu  \big)b_n,\\
         2(2n-1) b_{n-1} - (\mu+1)b_{n-1} &= (4n -3 -\mu)b_{n-1},
    \end{aligned}
    \end{equation*}
    we deduce that the power series $R_{\mu, 2}$ is a local solution of \Cref{eq:R-prop} if and only if 
    \begin{equation*}
        (6\mu b_1+(\mu^2-3\mu)b_0 ) \eta  + 
        \sum_{n = 1}^\infty 
        \Big\{ 
            2\mu (2n+3)(n+1) b_{n+1}- \big[  4n^2+(4\mu-2)n- \mu^2 + 3\mu \big]b_n
            +
            (4n -3 -\mu)b_{n-1}
        \Big\}\eta^{2n+1} = 0.
    \end{equation*}
    Since $2\mu (2n+3)(n+1)= 2\mu(2n^2 + 5n + 3)$, this holds if and only if the sequence $(b_n)_{n \in \mathbb{N}}$ satisfies the recursive relations specified in \eqref{eq:cond-bn-prop-R}.
\end{proof}
\noindent We conclude this section with the proof of \Cref{cor:special-solutions-of-Y}.
\begin{proof}[Proof of \Cref{cor:special-solutions-of-Y}] It follows from the definition \eqref{def:kummer's-function} of the Kummer's function $\mathcal M$ and the following relation on its derivative:
\begin{equation}\label{eq:derivative-of-Kummer-in-cor}
\mathcal M (a, c, \zeta) = \sum_{n = 0}^\infty \frac{(a)_n}{(c)_n}\frac{\zeta^n}{n!}, \qquad 
\frac{d}{d\zeta}M (a, c, \zeta) = \frac{a}{c} M (1+a, 1+c, \zeta),\qquad \text{for all } \zeta \in \mathbb C.
\end{equation}
Indeed, we can expand the power series $R_{\mu, 1} $ defined in \eqref{def:power-series-in-prop-R} as follows:
\begin{equation*}
\begin{aligned}
    R_{\mu, 1}(\eta) 
    &= 
    \mu \sum_{n = 0}^\infty \frac{1}{n!}
        \frac{\left( -\frac{\mu+1}{4}\right)_n}{\left( \frac{1}{2}\right)_n}\eta^{2n}
        -
        \sum_{n = 0}^\infty \frac{1}{n!}
        \frac{\left( -\frac{\mu+1}{4}\right)_n}{\left( \frac{1}{2}\right)_n} \eta \frac{d}{d\eta}(\eta^{2n})\\
    & =
   \mu\, 
            \mathcal M\Big( -\frac{\mu+1}{4},\frac{1}{2},\eta^2\Big)
            -
            \eta \frac{d}{d\eta}
            \left[ 
                \mathcal M\Big( -\frac{\mu+1}{4},\frac{1}{2},\eta^2\Big)
            \right]
\end{aligned}
\end{equation*}
Hence, applying the relation on the derivative of the Kummer's function as in \eqref{eq:derivative-of-Kummer-in-cor}, we obtain that
\begin{align*}
    R_{\mu, 1}(\eta) 
    & =
    \mu\, 
            \mathcal M\Big( -\frac{\mu+1}{4},\frac{1}{2},\eta^2\Big)
            -
            2 \eta^2 \frac{-\frac{\mu+1}{4}}{\frac{1}{2}}
                \mathcal M\Big( 1-\frac{\mu+1}{4},1+\frac{1}{2},\eta^2\Big) 
    \\
     &=
    \mu\, \mathcal M\Big( -\frac{\mu+1}{4},\frac{1}{2},\eta^2\Big)
    +  (\mu+1)\eta^2 
                \mathcal M\Big( \frac{3-\mu}{4},\frac{3}{2},\eta^2\Big). 
\end{align*}
Similarly, we can develop the power series $R_{\mu, 2} $ defined in \eqref{def:power-series-in-prop-R} as follows:
\begin{equation*}
\begin{aligned}
    R_{\mu, 2}(\eta) 
    &= 
    \eta+ 
    \frac{1}{4}
    \sum_{n = 1}^\infty \frac{2n+1-\mu}{n!}
        \frac{\left( \frac{5-\mu}{4}\right)_{n-1}}{\left( \frac{3}{2}\right)_n}\eta^{2n+1}
    \\
    &=
    \eta+ 
    \sum_{n = 1}^\infty
    \frac{1}{n!} \frac{1-\mu}{4}
    \frac{\left( 1+\frac{1-\mu}{4}\right)_{n-1}}{\left( \frac{3}{2}\right)_n}\eta^{2n+1}
    +
    \frac{1}{2}
    \sum_{n = 1}^\infty
    \frac{1}{(n-1)!}
    \frac{\left( \frac{5-\mu}{4}\right)_{n-1}}{\left( \frac{5}{2}-1\right)_n}
    \eta^{2n+1}
    \\
    &=
    \eta+ 
    \sum_{n = 1}^\infty
    \frac{1}{n!}
    \frac{\left( \frac{1-\mu}{4}\right)_{n}}{\left( \frac{3}{2}\right)_n}\eta^{2n+1}
    +
    \frac{1}{2}
    \sum_{n = 1}^\infty
    \frac{1}{(n-1)!}
    \frac{\left( \frac{5-\mu}{4}\right)_{n-1}}{\frac{3}{2}\cdot \left( \frac{5}{2}\right)_{n-1}}
    \eta^{2n+1}
\end{aligned}
\end{equation*}
Applying a substitution to the second series with $k = n-1$, we further obtain that
\begin{align*}
    R_{\mu, 2}(\eta) 
    &=
    \eta
    \sum_{n = 0}^\infty
    \frac{1}{n!}
    \frac{\left( \frac{1-\mu}{4}\right)_{n}}{\left( \frac{3}{2}\right)_n}\eta^{2n}
    +
    \frac{\eta^3}{3}
    \sum_{k = 1}^\infty
    \frac{1}{k!}
    \frac{\left( \frac{5-\mu}{4}\right)_{k}}{\cdot \left( \frac{5}{2}\right)_{k}}
    \eta^{2k}
    \\
    &
    =
    \eta
    \mathcal M\Big( \frac{1-\mu}{4},\frac{3}{2},\eta^2\Big)
    + 
    \frac{\eta^3}{3}
    \mathcal M\Big( \frac{5-\mu}{4},\frac{5}{2},\eta^2\Big).
\end{align*}
This concludes the proof of \Cref{cor:special-solutions-of-Y}.
\end{proof}

\section{Proof of \Cref{thm:TheoremSeparatedSolutions}}\label{sec:representation-via-harmonic-oscillator}
\noindent

\noindent The following section is devoted to the proof of \Cref{thm:TheoremSeparatedSolutions}, highlighting the deep connection between the Prandtl eigenvalue problem and the harmonic oscillator. We briefly recall from \Cref{sec:equivalent-forms-of-the-linearised-Prandtl-equations} that the variable $\eta \in \mathbb{C}$ and the constant $\mu \in \mathbb{C}$ are defined as
\begin{equation*}
    \eta = e^{\mp \frac{\pi}{8}\im }\sqrt[4]{|\beta||k|}(y-a) \in \mathbb{C},
    \qquad 
    \mu = 
    -\bigg(
         \frac{\sigma}{|\beta|} 
         \pm 
         \im 
         \alpha \sqrt{\frac{|k|}{|\beta|}}
        \bigg)
        \exp \Big(  \pm \frac{\pi}{4}\im \Big) \in \mathbb{C}.
\end{equation*}
In terms of $\eta \in \mathbb{C}$, the two functions 
$u(t,x,y) = \phi'_k(y) e^{ikx + \sigma \sqrt{|k|}t}$ and  
$v(t,x,y) = -ik \phi_k(y) e^{ikx + \sigma \sqrt{|k|}t}$  
satisfy the Prandtl equations \eqref{eq:linearised-Prandtl} (not necessarily with boundary conditions) if and only if $\Upsilon(\eta) := \phi_k (a + e^{\pm \frac{\pi \im }{8}}\eta/\sqrt[4]{|\beta||k|} )$
is a solution of the following Prandtl eigenproblem (cf.~\eqref{eq:equation-of-Upsilon-intro} in \Cref{sec:equivalent-forms-of-the-linearised-Prandtl-equations}): 
\begin{equation}\label{eq:ReducedPrandtlEP}
        -\Upsilon'''(\eta)+\eta^2\Upsilon'(\eta) -2\eta \Upsilon(\eta) = \mu \Upsilon'(\eta), \qquad \eta \in \mathbb{C}.
\end{equation}
Recalling that the definitions of $\psi_{\mu,1}$, $\psi_{\mu,2}$, $\Upsilon_{\mu, \eta_*, 1}$, and $\Upsilon_{\mu, \eta_*, 2}$ are given in \eqref{eq:main-thm-harmonic-oscillator}, \eqref{eq:Upsilon1}, and \eqref{eq:Upsilon2}, the first part of \Cref{thm:TheoremSeparatedSolutions}, in the absence of boundary conditions, follows from determining the general solution of \eqref{eq:ReducedPrandtlEP}. In the following proposition, we established in particular three linearly independent solutions.
\begin{prop}\label{prop:RepresentationByPsi}
    Let $\mu \in \mathbb C$ and  $\eta_\ast \in \mathbb C$ be two arbitrary constants. Then all solutions $\Ups:\CC \to \mathbb C$ of \eqref{eq:ReducedPrandtlEP} are entire and given by 
    \begin{equation*}
        \Ups (\eta) = c_0(\eta^2 - \mu) + c_1 \Ups_{\mu,\eta_*,1} (\eta) + c_2 \Ups_{\mu,\eta_*,2}(\eta)
    \end{equation*}
    for some constants $c_0,c_1,c_2 \in \CC$. 
    Here, the functions $\Ups_{\mu,i}: \CC \to \CC$ are given by 
   \begin{align*}
       \Ups_{\mu,i}(\eta)  := \Ups_{\mu,\eta_*,i}(\eta) = 
       \begin{cases} 
            \displaystyle \int_{\eta_*}^\eta  \left( 1+ \frac{\eta^2-\xi^2}{2} \right) \psi_{\mu, i } (\xi) \dd \xi -
             \frac{\psi'_{\mu,i}(\eta_*)}{2}, & \text{if } (\mu,i) \neq (1,2),\\
             \displaystyle \int_{\eta_*}^\eta  \left( 1+ \frac{\eta^2-\xi^2}{2} \right) g(\xi)\dd \xi -
             \frac{ g'(\xi)}{2}
             \quad \text{with }g(\xi) = e^{\frac{\xi^2}{2}} \erf (\xi), & \text{if } (\mu, i) =(1,2),
       \end{cases}
    \end{align*}
    for $i=1,2$.
\end{prop}

\begin{proof}[Proof of Proposition \ref{prop:RepresentationByPsi}] 
    First, one easily checks that $\Upsilon(\eta)= \mu - \eta^2$ solves \eqref{eq:ReducedPrandtlEP}, which leaves the remaining two linearly independent solutions to be found. Next, we show that any linear combination $\Ups := c_1 \Ups_{\mu, \eta_*, 1}+ c_2 \Ups_{\mu, \eta_*, 2}$ satisfies \eqref{eq:ReducedPrandtlEP}, where we recall that for any $i = 1,2$
    \begin{equation}\label{eq:RepresentationUps}
    \left\{
    \begin{alignedat}{8}
         &\Ups_{\mu,\eta_*,i}(\eta) &&:= \frac12 \Bcal_\mu \int_{\eta_*}^\eta \psi_{\mu,i}(\xi) \dd \xi\qquad
         &&&&\text{if }(\mu,i) \neq (1,2),\\
         &\Ups_{1,\eta_*,2}(\eta) &&:= \frac12 \Bcal_\mu \int_{\eta_*}^\eta e^{\frac{\xi^2}{2}}\erf(\xi) \dd \xi\qquad
         &&&&\text{otherwise},
    \end{alignedat}
    \right.
     \end{equation}
    with $ \Bcal_\mu:=-\frac{\dd^2}{\dd \eta^2} + \eta^2 -\mu$ the Schr\"odinger operator. Recall that $\psi_{\mu,1}$ is even and $\psi_{\mu,2}$ is odd. Because of \eqref{eq:RepresentationUps}, we distinguish two cases: {\it Part $(i)$} with $\mu \neq 1$ and {\it Part $(ii)$} with $\mu = 1$.

    \smallskip 
    \noindent 
    {\it Part $(i)$}: We address first the case $\mu \neq 1$ and we observe that $\Ups_{\mu,\eta_*, 1}$ is always the sum of an odd function and a multiple of $\eta \mapsto \eta^2-\mu$ and $\Ups_{\mu,\eta_*, 2}$ is always even. 
    The linear combination  $\Ups = c_1 \Ups_{\mu, \eta_*, 1}+ c_2 \Ups_{\mu, \eta_*, 2}$ satisfies 
    \begin{align*}
        -\Ups'''&+(\eta^2- \mu) \Ups'-2\eta \Ups 
        = 
        \Bcal_\mu \Ups' -2 \eta \Ups
        =
        \Bcal_\mu \frac{d}{d\eta} 
        \left( \frac 12 \Bcal_\mu   \int_{\eta_*}^\eta \psi \dd \xi \right) -\eta \Bcal_\mu \int_{\eta_*}^\eta \psi \dd \xi\\
        &
        =
        \frac 12   \Bcal_\mu^2 \frac{d}{d\eta} 
        \left(   \int_{\eta_*}^\eta \psi \dd \xi \right) 
        +
        \frac 12   \Bcal_\mu \left[ \frac{d}{d\eta} , \Bcal_\mu \right]
        \left(   \int_{\eta_*}^\eta \psi \dd \xi \right)-\eta \Bcal_\mu \int_{\eta_*}^\eta \psi \dd \xi
        \\
        &
        =
        \frac 12   \Bcal_\mu^2  \psi  
        +
        \frac 12   \Bcal_\mu 
        \left(2\eta \int_{\eta_*}^\eta \psi \dd \xi \right)
        -\eta \Bcal_\mu \int_{\eta_*}^\eta \psi \dd \xi
        \\
        &
        =
        \frac 12   \Bcal_\mu^2  \psi  
        +
        \eta 
        \Bcal_\mu 
        \int_{\eta_*}^\eta \psi \dd \xi
        +
        \left[ \Bcal_\mu,  \eta \right]
        \int_{\eta_*}^\eta \psi \dd \xi
        -\eta \Bcal_\mu \int_{\eta_*}^\eta \psi \dd \xi\\
        & =
         \frac 12  
         \Bcal_\mu^2  \psi 
         -
         2 \frac{d}{d\eta} \int_{\eta_*}^\eta \psi \dd \xi
         =
         \frac 12  
         \Bcal_\mu^2  \psi 
         -
         2  \psi,
    \end{align*}
    with $\psi = c_1 \psi_{\mu,1} + c_2\psi_{\mu,2}$.  With the help of the latter identities and $-\psi''+\eta^2 \psi =(\mu+2) \psi$, we compute
    \begin{align*}
        -\Ups'''+(\eta^2- \mu) \Ups'-2\eta \Ups 
         &= 
         \frac 12 \bigg(- \frac{d^2}{d\eta^2}+ \eta^2-\mu  \bigg)^2 \psi - 2 \psi \\
        &= 
         \bigg(- \frac{d^2}{d\eta^2}+ \eta^2-\mu  \bigg) \psi  - 2 \psi 
         = 2\psi -2 \psi = 0.
    \end{align*}
   Hence we have shown that any $\Ups = c_1 \Ups_{\mu, \eta_*, 1}+ c_2 \Ups_{\mu, \eta_*, 2}$ solves indeed \eqref{eq:ReducedPrandtlEP}. Next, we show that $\Ups_{\mu, \eta_*, 1}$ and $\Ups_{\mu, \eta_*, 2}$ are actually linearly independent (in particular non-vanishing). First, we see that $\Ups_{\mu, \eta_*, 1}\neq \Ups_{\mu, \eta_*, 2}$ by a comparison of parities. Now if one of these solutions vanished, we would have
    $$  \left( -\frac{\dd^2}{\dd \eta^2} + \eta^2-\mu \right) \int_{\eta_*}^\eta \psi_{\mu,i}(\xi) \dd \xi =0.$$
   In this case, the integral expression must be an eigenfunction of the Schr\"odinger operator
    \begin{align*}
        \int_{\eta_*}^\eta \psi_{\mu,1} \dd \xi & =\tilde c_1  \psi_{\mu-2,1} + \tilde c_2 \psi_{\mu-2,2},\\
        \int_{\eta_*}^\eta \psi_{\mu,2} \dd \xi &=\hat  c_1  \psi_{\mu-2,1} + \hat  c_2 \psi_{\mu-2,2}, 
    \end{align*}
    for some $\tilde c_1, ~\tilde c_2, ~\hat c_1, ~\hat c_2 \in \CC$. Thus, the parities show that $\tilde c_1 = \bar c_2=0$ since a non-vanishing  constant function is never an algebraic eigenfunction of the Schr\"odinger operator. Next, we distinguish two cases: If $\mu=-1$, we have (up to multiplicative constants)
        \begin{align*}
            \psi_{\mu,1}(\eta)&= e^{- \eta^2/2}, && \psi_{\mu,2}(\eta) = e^{-\eta^2/2} \int_{\eta_*}^\eta e^{\xi^2 } \dd \xi,\\
            \psi_{\mu-2,2}&= e^{\eta^2/2} \int_{\eta_*}^\eta e^{-\xi^2}\dd \xi , && \psi_{\mu-2,1}(\eta) = e^{\eta^2/2}.
        \end{align*}
    Taking a derivative, one easily checks that the above equations are never satisfied for any choice $\tilde c_2 , ~ \bar c_1 \in \CC$. If $\mu\in \mathbb C \setminus \{ 1,- 1\}$, taking a derivative of the equations and using the up-operator $\mathcal{A}_\uparrow:=(\eta-\frac{\dd}{\dd \eta})$ and $A_\uparrow \psi_{\mu-2}=\psi_{\mu} \neq 0$ which holds for any $\mu \in \CC\setminus\{1,-1\} $, we can rewrite this equation equivalently by
    \begin{align*}
        (\tilde c +1)\psi_{\mu-2}' =\eta \psi_{\mu-2}, \quad \psi_{\mu-2}(\eta_*)=0.
    \end{align*}
    The unique solution is, however, $\psi_{\mu-2}\equiv0$ which implies $\psi_{\mu}\equiv0$. Hence, we have found that $\Upsilon_{\mu, \eta_*, 1}$ and $\Upsilon_{\mu, \eta_*, 2}$ are two non-vanishing solutions, one odd plus a quadratic function and the second one even. In sum, they are two linearly independent solutions  of \eqref{eq:ReducedPrandtlEP}. 
    
    \noindent 
    In order to conclude the linear independence of all three solutions (the third one being $\eta \mapsto \mu-\eta^2$) we need to show that the resulting even solution satisfies
    \begin{align*}
        \Bcal_\mu \int_{\eta_*}^\eta \psi_{\mu,2} \dd \xi \neq c (\eta^2- \mu)  \quad \Leftrightarrow \quad \Bcal_\mu \left[ \int_{\eta_*}^\eta \psi_{\mu,2} \dd \xi -c \right] \neq 0
    \end{align*}
    for any $c \in \CC$. Similar to above, if this was not the case, we would have $\int_{\eta_*}^\eta \psi_{\mu,2} \dd \xi -c =\tilde c \psi_{\mu-2,1}$ and equivalently 
    \begin{align*}
        (\tilde c+1)\psi_{\mu-2,1}' = \eta \psi_{\mu-2,1}, \quad \psi_{\mu-2,1}(\eta_*) = -c.
    \end{align*}
    Here the unique solution is $\tilde c=-2,~\mu=1$ and $\psi_{-1,1}(\eta) =-ce^{-(\eta^2-\eta_*^2)/2} $. Thus, for $\mu \neq 1$, the three found solutions $(\mu-\eta^2)$, $\Upsilon_{\mu, \eta_*, 1}$ and $\Upsilon_{\mu, \eta_*, 2}$ are linearly independent. 

    \smallskip 
    \noindent
    {\it Part $(ii)$ }: to conclude the proof of \Cref{prop:RepresentationByPsi}, we need to consider the case of $\mu = 1$. Here, we have a more direct approach, since the functions $\Upsilon_{1, \eta_*, 1}$ and $\Upsilon_{1, \eta_*, 2}$ in \eqref{eq:RepresentationUps} can be determined explicitly: 
    \begin{equation*}
        \Upsilon_{1, \eta_*, 1} = \frac12 \Bcal_\mu \int_{\eta_*}^\eta 
        \left(
        e^{\frac{\xi^2}{2}} 
        -
        \sqrt{\pi}
        \xi 
        \erf \im(\xi)
        \right)
        \dd \xi,
        \qquad 
        \Upsilon_{1, \eta_*, 2}
        :=
        \frac12 \Bcal_\mu \int_{\eta_*}^\eta 
        e^{\frac{\xi^2}{2}}\erf(\xi)
        \dd \xi
    \end{equation*} 
    The fact that $\mu- \eta^2$, $\Upsilon_{1, \eta_*, 1}$ and $\Upsilon_{1, \eta_*, 2}$ are linearly independent follows from a similar Ansatz as in {\it Part~$(i)$}. It remains to show that $\Upsilon_{1, \eta_*, 2}$ is a solution for \eqref{eq:ReducedPrandtlEP} with $\mu = 1$. We remark that $g(\eta):=  e^{\frac{\eta^2}{2}} \erf(\eta)$ satisfies $- g''+ \eta^2 g = -g \Rightarrow \Bcal_{1} g = -2 g $. Hence, dropping the indexes in 
    $\Upsilon = \Upsilon_{1, \eta_*, 2}$, we gather that
    \begin{align*}
        -\Ups'''+\eta^2\Ups'-2\eta \Ups- \Ups'&=   \frac 12  
         \Bcal_1^2  g
         -
         2  g  = 2 g -2g = 0.
    \end{align*}
    This concludes the proof of \Cref{prop:RepresentationByPsi}.
\end{proof}
\noindent 
It is obvious that the constants $c_0,c_1$ and $c_2$ above can always be chosen in such a way that $\Upsilon $ satisfies the boundary conditions of the Prandtl equation $\Upsilon (\eta_*)=\Upsilon'(\eta_*)=0$ for an arbitrary $\eta_*\in \CC$. 
\begin{prop}
    Under the conditions of \Cref{prop:RepresentationByPsi}, the boundary conditions $\Ups(\eta_*) = \Ups'(\eta_*)$ are satisfied if the triple triple $(c_0,\,c_1,\,c_2)\in \mathbb C^3$ is a solution of the linear system
   \begin{equation}\label{eq:prop-constants-for-bdycdt}
    \begin{pmatrix}
        \mu-\eta_*^2 & -\frac 12 \psi_{\mu,1}'(\eta_*) 
        &  -\frac 12 \psi_{\mu,2}'(\eta_*) \\[0.5em]
        -2\eta_* & \psi_{\mu,1}(\eta_*) & \psi_{\mu,2}(\eta_*)
    \end{pmatrix}
    \raisebox{-0.8em}{
    $
    \begin{pmatrix}
        c_1 \\[0.8em] c_2 \\[0.8em] c_3
    \end{pmatrix}
    $
    }
    =
    \begin{pmatrix}
        0 \\[0.5em] 0
    \end{pmatrix},\qquad \text{if } \mu \neq 1,
\end{equation}
while, in the case $\mu = 1$, $(c_0,\,c_1,\,c_2)\in \mathbb C^3$ satisfies \eqref{eq:prop-constants-for-bdycdt} with $\psi_{\mu,2}(\eta)$ replaced by $g(\eta) = e^{\frac{\eta^2}{2}} \erf(\eta)$.
\end{prop}
\begin{proof}
By Proposition \ref{prop:RepresentationByPsi}, the condition $\Ups_{\mu}(\eta_*) =0$ immediately implies the first row of the linear system whereas $\Ups'(\eta_*)=0$ leads to the second one.
\end{proof}

\section{Proof of Corollary \ref{cor:GrowthOnR}: the behaviour at infinity}

\noindent
An important aspect of $\Upsilon$ in \Cref{prop:RepresentationByPsi} is its behavior at infinity, which determines the asymptotics of a general solution $(\tau, W)$ to \eqref{eq:W}.  From \Cref{tab:variables_parameters_functions}, for positive frequencies $k \in \mathbb{N}$, we recall the relations $\eta = e^{-\frac{\pi  }{8}\im} z$, $\tau = \mu e^{\frac{\pi  }{4}\im}$, and $\Upsilon(\eta) = (\tau - z^2) W(z)$. Thus, $(\mu, \Upsilon)$ solves \eqref{eq:ReducedPrandtlEP} if and only if $(\tau, W)$ solves \eqref{eq:W}.

\noindent
We have established that for $\tau = e^{\frac{5 \pi \im}{4}}$ (i.e., $\mu = -1$), there exists an explicit solution $W$ satisfying \Cref{crit:spectral-condition-W} (cf.~\Cref{cor:main-result-tau-W-unique} and \eqref{eq:Upsilon1-associated-to-W-criterium} for $\Upsilon$), with the boundary conditions  
 $\lim\limits_{z \to -\infty} W(z) = 0$ and $\lim\limits_{z \to +\infty} W(z) = 1$, namely
 \begin{equation*}
        \lim_{\eta  \to \infty_\mathbb C, \; \arg(-\eta) = -\frac{\pi }{8}\im} \Upsilon(\eta) = 0
        \qquad 
        \lim_{\eta  \to \infty_\mathbb C, \; \arg(\eta) = -\frac{\pi }{8}\im}
        \Upsilon(\eta) = 1.
\end{equation*}
The question now is whether another pair $(\tau, W)$ with ${\rm Im }\, \tau < 0$ (i.e.~${\rm Re}\,\mu < {\rm Im} \,\mu$) satisfies the same criterion. The answer is negative.

\smallskip 
\noindent 
The key idea is that if $\mu \neq 2n-1$ for any $n \in \mathbb{N}_0 \setminus \{1\}$ (i.e., $\tau \neq (2n-1) e^{\frac{\pi }{4} \im}$), then any $\Upsilon$ from \Cref{prop:RepresentationByPsi} with $\eta_* = 0$ and  
\begin{equation}\label{eq:sec-5-general-Ups}
    \Upsilon(\eta) = c_0 (\mu- \eta^2) + c_1 \Upsilon_{\mu,0,1}(\eta)
    + c_2 \Upsilon_{\mu,0,2}(\eta)\qquad \text{with}\quad (c_1, c_2) \neq (0,0),
\end{equation}
exhibits exponential growth in at least one of the sectors $|\arg(\eta)| < \pi/4$ or $|\arg(-\eta)| < \pi/4$. Note that if $(c_1, c_2) = (0,0)$, then $\Upsilon$ is purely quadratic, and $W(z)$ is constant, meaning it cannot satisfy \Cref{crit:spectral-condition-W}. \Cref{cor:main-result-tau-W-unique} thus requires $\mu = -1$. In this case, $\Upsilon_{-1,0,2}$ grows exponentially in both sectors $|\arg(\eta)| < \pi/4$ and $|\arg(-\eta)| < \pi/4$, while $\Upsilon_{-1,0,1}$ exhibits quadratic growth on one side and exponential decay on the other (cf.~\Cref{lemma:AsymptoticsMu=-1}). The boundary conditions of $W(z)$ in \Cref{crit:spectral-condition-W} then impose $c_2 = 0$, leaving only two degrees of freedom in $c_0$ and $c_1$, ensuring the uniqueness of $W$.

\medskip 
\noindent 
In what follows, we formalize the above heuristics. For simplicity, we exclude the case $\mu = 1$ as it does not satisfy ${\rm Re}\, \mu < {\rm Im}\, \mu$. Recall that for any $\mu \in \mathbb{C} \setminus \{1\}$, the functions $\Upsilon_{\mu,0,i}(\eta)$ in \eqref{eq:sec-5-general-Ups} for $i = 1,2$ are defined as
\begin{equation}\label{eq:last-sec-Ups-psi}
    \Upsilon_{\mu,0,i}(\eta) = \int_0^\eta \left( 1 +  \frac{\eta^2-\xi^2}{2} \right)\psi_{i, \mu}(\xi) d\xi 
    \quad \text{with }
    \left\{
    \begin{alignedat}{8}
        \psi_{\mu,1} (\eta) 
        &=
        \mathcal M \Big( -\frac{1+\mu}{4}  , \frac{1}{2}, \eta^2 \Big)
        e^{-\frac{\eta^2}{2}},
        \\
        \psi_{\mu,2} (\eta)
        &= 
        \eta   \,
        \mathcal M \Big(\; \frac{1-\mu}{4}, \frac{3}{2}, \eta^2 \, \Big)
        e^{-\frac{\eta^2}{2}}. 
    \end{alignedat}
    \right.
\end{equation}
The asymptotics of  $ \Upsilon_{\mu,0,i}$ are therefore related to the asymptotics of Kummer's Hypergeometric functions, of which we repeatedly make use throughout this section:
\begin{prop}\label{eq:AsymptoticsKummerFunction}
    For every ${\rm a},{\rm c} \in \mathbb C$ with ${\rm a}, {\rm c-a} \neq -2n$ for any $n \in \mathbb N_0$, the following asymptotic expansion holds true:
    \begin{align*}   
       \MB\left({\rm a},{\rm c}, \zeta\right) \quad \sim \quad \frac{\Gamma(\rm c)}{\Gamma(\rm a)}  e^\zeta \zeta^{{\rm a}-{\rm c}} \left( 1+  (1-{\rm a})({\rm c}-{\rm a})\frac{1}{\zeta}+\Ocal\left(\frac{1}{\zeta^{2}}\right) \right)
    \end{align*}
    as $\zeta \to \infty_\mathbb{C} $ with $|\arg(\zeta)|< \pi/2-\delta$ for any $\delta >0$. 
\end{prop}
\begin{proof}
  This is a classical result (see, e.g., Section 10.1 in \cite{OLVER1974229}), noting that we have included only the leading term determined by the branch $|\arg(\zeta)| < \pi/2 - \delta$. Additionally, because of this branch, the asymptotics is single valued.
\end{proof}
\noindent
Ultimately, we are only interested in $\zeta = \eta^2 = e^{\pm \frac{\pi}{4}\im}z^2$ as $z \to \pm \infty$ in $\mathbb R$ and we recall that for given functions $h, \phi: \mathbb{R} \to \mathbb{C}$, the Poincaré expansion $h(z) \sim \phi(z)$ as $z \to \pm \infty$ means  
\begin{equation*}
    h(z)-\phi(z) = o(\phi(z)) \text{ as }z \to \pm \infty.
\end{equation*}
By setting $\zeta = \eta^2 $ in \Cref{eq:AsymptoticsKummerFunction}, we remark that, for the most values of $\mu$, $\psi_{\mu,1}$ and $\psi_{\mu,2}$ behave asymptotically as $e^{\eta^2/2}$. The main technical challenge is extending this behaviour to the integrals of $\Upsilon_{\mu,0,i}$. In essence, we seek to avoid pathological examples such as  
$$ f(z) = \sin(\exp(z^2)) \quad \Rightarrow \quad f'(z) =2 z e^{z^2} \cos (\exp(z^2)) $$
showing that the primitive of an asymptotically exponentially growing function does not need to be unbounded (on the real line $z \in \mathbb R$). However, the idea is that the oscillations of the hypergeometric functions are not sufficiently large in order to compensate the exponential growth of the derivative.   

\medskip 
\noindent 
Since in \eqref{eq:last-sec-Ups-psi} we have $({\rm a}, {\rm c}) = (-(1+\mu)/4,2)$ or $({\rm a}, {\rm c})= ((1-\mu)/4,3/2)$, \Cref{eq:AsymptoticsKummerFunction} does not apply to $\mu \in 2\mathbb{Z}-1$, of which only $\mu \in -2\mathbb{N}-1$ is relevant to our analysis. Hence we work two separated cases: first $\mu \notin2\ZZ-1$ (\Cref{lemma:NonOddMuAsymptotics}) and afterwards proceed by a semi-explicit representation in case $\mu$ is odd and negative (\Cref{lemma:SecondBranchSolutionsHO} and \Cref{lemma:AsymptoticsOddMu1}). 
In both we need a simple auxiliary lemma:
\begin{lemma}   \label{lemma:helpfulAsymptotics}
    Let $\gamma\in \CC$, $b \in \mathbb C\setminus\{ 0 \}$ with  $\operatorname{arg}(b) \in ]-\pi/4, \pi/4[$ and $\operatorname{arg}(b^2) \in ]-\pi/2, \pi,2[$. 
    Then, for $z \in \mathbb R$ the following asymptotics hold true:
    \begin{alignat*}{16}
        &\int_{{\rm sgn}(z)}^z \exp \left( b^2 \frac{\tilde z^2}{2} \right) (b^2 \tilde z^2)^\gamma \dd \tilde z \quad &&\sim \quad  \frac1b 
        \exp \left( b^2 \frac{z^2}{2} \right) \frac{(b^2z^2)^{\gamma}}{bz} ,\qquad 
        &&&&\text{as } z \to \pm \infty,\\
        &
        \int_{{\rm sgn}(z)}^z
        \exp \left( b^2 \frac{\tilde z^2}{2} \right) (b^2 \tilde z^2)^{\gamma}(b \tilde z) \dd \tilde z \quad &&\sim \quad  \frac1b 
        \exp \left( b^2 \frac{z^2}{2} \right) (b^2z^2)^{\gamma},\qquad 
        &&&&\text{as } z \to \pm \infty.
    \end{alignat*}
\end{lemma}
\noindent We start the integrals from $\tilde{z} = {\rm sgn}(z) = \pm 1$ to avoid the possible singularity at $\tilde{z} = 0$, though, in reality the considered functions in the next lemmas are entire, making this a mere technicality. We note moreover that ${\rm arg}(b^2) \in ]-\pi/2, \pi/2[$ is set, thus $(b^2 z^2)^{\gamma}$ is single-valued as  
$(b^2 z^2)^{\gamma} = (|b|^2z^2)^\gamma e^{\im \operatorname{arg}(b^2)}$.
While the lemma extends to other values of $b$ (with possible multi-valued asymptotics), this is beyond our scope. 
\begin{proof}
    The result follows from applying integration by parts twice. We carry it out only for the first integral, as the second follows an identical procedure.
    We have
        \begin{align*}
        \int_{{\rm sgn}(z)}^z \exp \left( b^2 \frac{\tilde z^2}{2} \right) &(b^2 \tilde z^2)^\gamma \dd \tilde z 
        =
        \int_{{\rm sgn}(z)}^z \frac{d}{d\tilde z} 
        \left( 
            \exp \left( b^2 \frac{\tilde z^2}{2} \right)
        \right)\frac{1}{b}\frac{(b^2 \tilde z^2)^{\gamma}}{b \tilde z} \dd \tilde z
        \\
        &
        = \frac 1b 
        \exp \left( b^2 \frac{\tilde z^2}{2} \right) \frac{(b^2  z^2)^{\gamma}}{b z} \bigg|_{{\rm sgn}(z)}^{z} + 
        (2\gamma-1)
        \int_{{\rm sgn}(z)}^z 
        \exp \left( b^2 \frac{\tilde z^2}{2} \right)  (b^2\tilde z^2)^{\gamma-2} \dd  \tilde z\\
        &
        = 
        \frac 1b 
        \exp \left( b^2 \frac{ z^2}{2} \right) \frac{(b^2 z^2)^{\gamma}}{bz}
        -\frac 1b 
        \exp \left( \frac{b^2}{2} \right)\frac{ (b^2)^{\gamma}}{b} + 
        (2\gamma-1)
        \int_{{\rm sgn}(z)}^z 
        \exp \left( b^2 \frac{\tilde z^2}{2} \right)  (b^2\tilde z^2)^{\gamma-2} \dd  \tilde z.
    \end{align*}
    The first term on the right-hand side of the last identity is as in the claimed asymptotics. Additionally, the remaining terms are of order $o\left(
        \exp \left( b^2  z^2/2 \right) (b^2 z^2)^{\gamma}/(bz) \right) $: repeating the same integration by parts with $\gamma$ replaced by $\gamma-2$, the last integral develops as
    \begin{align*}
        \int_{{\rm sgn}(z)}^z 
        \exp &\left( b^2 \frac{\tilde z^2}{2} \right) 
         (b^2\tilde z^2)^{\gamma-2} 
        \dd  \tilde z 
        \\
        &=
        \frac 1b 
        \exp \left( b^2 \frac{ z^2}{2} \right) \frac{(b^2 z^2)^{\gamma-2}}{b z }
        -\frac 1b 
        \exp \left( \frac{b^2}{2} \right) \frac{b^{\gamma-2}}{b} + 
        (2\gamma-3)
        \int_{{\rm sgn}(z)}^z 
        \exp \left( b^2 \frac{\tilde z^2}{2} \right)  (b^2\tilde z^2)^{\gamma-4} \dd  \tilde z.
    \end{align*}
    The first term is $o\left(
        \exp \left( b^2  z^2/2 \right) (b^2 z^2)^{\gamma}/(bz) \right) $ as $z \to \pm \infty$. For the remaining term,
         we apply the mean value theorem: we know that
         $\R \ni z  \mapsto | e^{b^2z^2/2} (b^2 z^2 )^{\gamma-4}| =|(b^2)^{\gamma-4}| e^{z^2/(2\sqrt{2})} (z^2)^{{\rm Re}\gamma -4} $ is convex for sufficiently large $|z|$, saying $|z|>|\bar{z}|\gg 1$ with ${\rm sgn}(\bar{z})= {\rm sgn}(z)$. 
         Hence we can estimate the integral using the mean value theorem in the interval $[\bar{z}, z]$ or $[z, \bar{z}]$:
         \begin{align*}
            \left| \int_{{\rm sgn}(z)}^z 
        \exp \left( b^2 \frac{\tilde z^2}{2} \right)  (b^2\tilde z^2)^{\gamma-4} \dd \tilde z  \right| \leq 
        \int_{{\rm sgn}(z)}^{\bar z} 
        \left| 
        \exp \left( b^2 \frac{\tilde z^2}{2} \right)  (b^2\tilde z^2)^{\gamma-4}  \right| \dd \tilde z  + 
        |z-\bar{z}|
        \left| 
        \exp \left( b^2 \frac{ z^2}{2} \right)  (b^2 z^2)^{\gamma-4}
        \right|.
         \end{align*}
          The last term is
         $o\left(
        \exp \left( b^2  z^2/2 \right) (b^2 z^2)^{\gamma -1} \right) $, this concludes the proof of the lemma.
\end{proof}
\subsection{Asymptotic expansions for $\mu \notin 2\ZZ-1 $}$\,$

\noindent 
We apply the result to the integral of the  hypergeometric functions in the case $\mu\notin 2\ZZ-1$, where the asymptotic expansions of \Cref{eq:AsymptoticsKummerFunction} holds true.
\begin{lemma} \label{lemma:NonOddMuAsymptotics}
    Let $b = e^{\pm \im  \frac{ \pi}{8}}$ and set $\operatorname{arg}(b) := \pm \frac{   \pi}{8} ,\,\operatorname{arg}(b^2) := \pm \frac{\pi}{4}$. Consider $\mu \notin 2\ZZ-1$ and $m \in \{0, 1\}$. Then, the following asymptotics hold true as $z \to \pm \infty$:
    \begin{alignat*}{4}
            \psi_{\mu, 1}(b z)
            &\sim 
            \mathcal{C}_{\mu,1}
            \exp\left( b^2 \frac{z^2}{2}\right)
            \big( b^2 z^2 \big)^{-\frac{\mu+3}{4}},
            \qquad 
            &&\int_{0}^z (b^2 \tilde z^2)^m \psi_{\mu,1}(b \tilde z) \dd \tilde z 
            \sim 
            \frac{\Ccal_{\mu,1} }{b} 
            \exp\left( b^2 \frac{z^2}{2}\right) 
            \frac{\big( b^2 z^2 \big)^{m-\frac{\mu+3}{4}}}{bz} 
            \\
            \psi_{\mu, 2}(b z)
            &\sim 
            \mathcal{C}_{\mu,2}
            \exp\left( b^2 \frac{z^2}{2}\right)
            \big( b^2 z^2 \big)^{-\frac{\mu+5}{4}}
            (bz) 
            ,
            \qquad 
            &&\int_{0}^z 
            (b^2 \tilde z^2)^m
            \psi_{\mu,2}(b \tilde z) \dd \tilde z 
            \sim 
            \frac{\Ccal_{\mu,2} }{b} 
            \exp\left( b^2 \frac{z^2}{2}\right) 
            \big( b^2 z^2 \big)^{m-\frac{\mu+5}{4}} 
    \end{alignat*}
    with $\Ccal_{\mu,1}:=\frac{\Gamma\left(\frac{1}{2}\right)}{\Gamma\left(-\frac{\mu+1}{4}\right)}$
    and $\Ccal_{\mu,2}:=\frac{\Gamma\left(\frac{3}{2}\right)}{\Gamma\left(\frac{1-\mu}{4}\right)}$.
\end{lemma}
\begin{proof}
    The asymptotics of $\psi_{\mu,1}$ and $\psi_{\mu,2}$ are a direct application of \Cref{eq:AsymptoticsKummerFunction} together with \eqref{eq:last-sec-Ups-psi} and $\zeta =  \eta^2 = b^2 z^2$  (thus ${\rm arg}(\zeta) = \pm \pi/4 \in ]-\pi/2-\delta, \pi/2+\delta[$). We now evaluate the integrals, considering only $\psi_{\mu,1}$, as $\psi_{\mu,2}$ follows by an identical procedure. We define 
    \begin{alignat*}{8}
        f_1(z) := \Ccal_{\mu,1}\exp\left( b^2 \frac{z^2}{2}\right)
        \big( b^2 z^2 \big)^{m-\frac{\mu+3}{4}},
    \end{alignat*}
    Thanks to \Cref{lemma:helpfulAsymptotics} with $ b= e^{\pm \frac{\pi}{8}\im }$ and  $\gamma = -(\mu+3)/4$, we gather that 
    \begin{align*}
        \int_{{\rm sgn}(z)}^z f_1(\tilde z)d \tilde z
        \sim 
         \Ccal_{\mu,1}\exp\left( b^2 \frac{z^2}{2}\right)
        \frac{\big( b^2 z^2 \big)^{m-\frac{\mu+3}{4}}}{bz}=:h_1(z).
    \end{align*}
    From $(b^2 z^2)^m \psi_{\mu,1}\sim f_1$ we deduce that $(b^2 z^2)^m \psi_{\mu,1}-f_1$ is $o(f_1(z))$ as $z \to \pm \infty$. In particular, for any $\ee>0$ there exists $z_\ee \gg 1$ such that
    \begin{equation*}
        |(b^2 z^2)^m \psi_{\mu,1}(z)-f_1(z)|< \ee |f_1(z)|,\qquad \text{for any }|z|>z_\ee.
    \end{equation*}
    It follows the following identities
    \begin{align*}
       \left| \int_0^z (b^2 z^2)^m \psi_{\mu,1}(\tilde z)d\tilde z- h_1(z) \right|
       &\leq 
       \left| \int_0^z(b^2 z^2)^m  \psi_{\mu,1}(\tilde z)d\tilde z- \int_{{\rm sgn}(z)}^z f_1(\tilde z)d \tilde z \right|
       +o(h_1(z))
       \\
       &\leq 
      \ee \left| 
                \int_{{\rm sgn(z) }z_\ee}^z|f_1(\tilde z)| d\tilde z
            \right|
        +o(h_1(z))
        \\
        & 
        \leq
        \ee 
        |\mathcal{C}_{\mu,1}|\left|(b^2)^{-\frac{\mu+3}{4}}\right|
        \int_{z_\ee}^{|z|} e^{\frac{\tilde z^2}{2\sqrt{2}}} (\tilde z^2)^{m-\frac{{\rm Re}(\mu)+3}{4}}d \tilde z
        +o(h_1(z)). 
    \end{align*}
    Using \Cref{lemma:helpfulAsymptotics} with $ b = 1/{\sqrt{2\sqrt{2}}}$, ${\rm arg}(b) = 0$ and ${\rm arg}(b^2) = 0$ implies that following asymptotics:
    \begin{equation*}
        \int_{z_\ee}^{|z|} e^{\frac{\tilde z^2}{2\sqrt{2}}} \left( \frac{\tilde z^2}{2\sqrt{2}} \right)^{m-\frac{{\rm Re}(\mu)+3}{4}}d\tilde z 
        \sim  
        \sqrt{2\sqrt{2}} \, e^{\frac{z^2}{2\sqrt{2}}}  \left( \frac{ z^2}{2\sqrt{2}} \right)^{m-\frac{{\rm Re}(\mu)+3}{4}-1} =  \Ocal \left( h_1(z)\right).
    \end{equation*}
    We thus obtain that 
    \begin{equation*}
        \left| \int_0^z (b^2 z^2)^m \psi_{\mu,1}(\tilde z)d\tilde z- h_1(z) \right| 
        = 
        \ee  \Ocal \left( h_1(z)\right) + o \left( h_1(z)\right).
    \end{equation*}
    The claim follows from the arbitrariness of $\ee>0$.
\end{proof}
\noindent 
We can transfer the above asymptotics directly to $\Upsilon_{\mu, 0, 1}$ and $\Upsilon_{\mu, 0, 2}$ using \eqref{eq:last-sec-Ups-psi}: 
\begin{cor}\label{cor:final-asymptotic-Ups-mu-not-integer}
    Let $b = e^{\pm \im \frac{  \pi}{8}}$ and set $\operatorname{arg}(b) := \pm \frac{   \pi}{8} ,\,\operatorname{arg}(b^2) := \pm \frac{\pi}{4}$. Consider $\mu \notin 2\ZZ-1$. Then, the following asymptotics hold true as $z \to \pm \infty$:
    \begin{alignat*}{4}
            \Upsilon_{\mu,0,1}(b z)
            &\sim 
            \mathcal{C}_{\mu,1}
            \exp\left( b^2 \frac{z^2}{2}\right)
            \frac{\big( b^2 z^2 \big)^{-\frac{\mu+3}{4}}}{bz},\qquad 
            \Upsilon_{\mu,0,2}(b z)
            &\sim 
            \mathcal{C}_{\mu,2}
            \exp\left( b^2 \frac{z^2}{2}\right)
            \big( b^2 z^2 \big)^{-\frac{\mu+5}{4}},
    \end{alignat*}
    with $\Ccal_{\mu,1}:=\frac{\Gamma\left(\frac{1}{2}\right)}{\Gamma\left(-\frac{\mu+1}{4}\right)}$
    and $\Ccal_{\mu,2}:=\frac{\Gamma\left(\frac{3}{2}\right)}{\Gamma\left(\frac{1-\mu}{4}\right)}$.
\end{cor}

\subsection{Asymptotic expansions for $\mu \in -2\NN-1 $}$\,$

\smallskip 
\noindent 
We next address the values $\mu \in -2\NN-1$ and we begin expressing $\psi_{\mu, 1}$ and $\psi_{\mu, 2}$ with a semi-explicit formula.
\begin{lemma}   \label{lemma:SecondBranchSolutionsHO}
    Let $\mu = -(2m+3)\in -2\NN-1$ with $m \in \mathbb N_0$. Define the down-operator $\Acal_\downarrow:=\zeta + \frac{\dd}{\dd \zeta} $. Then,
    with indexes $i = (3 +(-1)^{m+1})/2 \in \{1, 2\}$ and $j = (3 +(-1)^{m})/2 \in \{1, 2\}$, one has
    \begin{align*} 
      \psi_{\mu,i}(\eta) &= \frac{\left \lceil{ \frac{m}{2} }\right \rceil!}{ \left(2\left \lceil{ \frac{m}{2} }\right \rceil\right)!} e^{-\eta^2/2} \frac{\dd^{m} }{\dd \eta^{m}} e^{\eta^2}
      ,
      \qquad 
      \qquad 
      \psi_{\mu,j}(\eta) = 
      \frac{1}{\left \lceil{ \frac{m-1}{2} }\right \rceil!}
      \frac{1}{4^{\left \lceil{ \frac{m-1}{2} }\right \rceil}}
      \Acal_\downarrow^m \left( e^{\eta^2/2} \int_{0}^\eta e^{-{\xi}^2} \dd \xi\right),
    \end{align*}
    for any $\eta \in \mathbb C$, where $\left \lceil{ \cdot  }\right \rceil $ is the ceiling function. In particular, it follows that there exist three polynomials $p_m,\,q_m,\, \tilde q \in \mathbb \R[\eta]$ of degree $m$ such that $p_m,\,q_m$ have leading-order coefficient $2^m \left \lceil{ \frac{m}{2} }\right \rceil!/\left(2\left \lceil{ \frac{m}{2} }\right \rceil\right)! $ and $2^m / ( \left \lceil{ \frac{m-1}{2} }\right \rceil!4^{\left \lceil{ \frac{m-1}{2} }\right \rceil}) $, respectively, as well as
    \begin{align}  \label{eq:SchrödingerSolutionsNegativeOdd1}
        \psi_{\mu,i}(\eta)& = p_m(\eta) e^{\eta^2/2}, \\
        \label{eq:SchrödingerSolutionsNegativeOdd2} 
        \psi_{\mu,j}(\eta) &= q_m(\eta)  e^{\eta^2/2} \int_{0}^\eta e^{-{\xi }^2} \dd \xi  + \tilde q(\eta) e^{-\eta^2/2}.
    \end{align}
\end{lemma}
\begin{proof}
   We consider only the case $m \in 2\mathbb{N}_0$, as $m \in 2\mathbb{N}_0 + 1$ can be treated similarly. If $m \in 2\mathbb N_0$, we have that $ i = 1 $ and $j = 2$ and the functions
   \begin{align*} 
      h_m(\eta) &:= e^{-\eta^2/2} \frac{\dd^{m} }{\dd \eta^{m}} e^{\eta^2}
      ,
      \qquad 
      \qquad 
      g_m(\eta) := 
      \Acal_\downarrow^m \left( e^{\eta^2/2} \int_{0}^\eta e^{-{\xi}^2} \dd \xi\right),
    \end{align*}
    are even and odd, respectively. The functions $h_m$, $g_m$ are solutions of the harmonic oscillator $-\psi''  + \eta^2 \psi  = (\mu+2)\psi   = - (2m + 1)\psi$ (cf.~\Cref{lemma:appx-solutions-harmonic-osc-}). By showing that 
    \begin{equation*}
         h_m(0) =\frac{m!}{\left( \frac{m}{2}\right)!}  
         \qquad 
          g_m'(0) 
        = 
        \left( \frac{m}{2} \right)!2^m,
    \end{equation*}
    then the statement is obtained by uniqueness of the solutions of the harmonic oscillator, since $\psi_{\mu,1}(0) = \psi_{\mu,2}'(0) = 1$. First, we have that
    \begin{align*}
        h_m(0) = 
        \frac{\dd^{m} }{\dd \eta^{m}}\Big|_{\eta = 0}
        \left(\sum_{k = 0}^\infty  \frac{\eta^{2k}}{k!}\right) = 
        \sum_{k = \frac{m}{2}}^\infty \frac{(2k)!}{(2k-m)!}\frac{\eta^{2k-m}}{k!}\Big|_{\eta = 0} = \frac{m!}{\left( \frac{m}{2}\right)!}.
    \end{align*}
    We show by induction the second identity: if $m = 0$ then  we have that $g_0(\eta)= e^{\eta^2/2} \int_{0}^\eta e^{-{\xi}^2} \dd \xi$, implying that $g_0'(0)= 1$. Assume now that $g_m'(0) =(m/2)!2^m$, for a given $m \in \mathbb N_0$. Then we remark that 
    \begin{align*}
        g_{m+2}(\eta) 
        &= \Acal_\downarrow^2 g_{m}(\eta) = (\eta^2+1) g_m(\eta) + 2\eta g_m'(\eta) + g_m''(\eta)\\
        &= (2\eta^2+2m+2) g_m(\eta) + 2\eta g_m'(\eta).
    \end{align*}
    Applying one derivative,  computing the result in $\eta = 0$ and recalling that $g_m$ is odd, we obtain
    \begin{align*}
        g_{m+2}'(0) 
        = (2m+4) g_m'(0) = 2^2 \frac{m+2}{2} 2^m \left(\frac{m}{2} \right)! =
        2^{m+2} \left(\frac{m+2}{2} \right)!.
    \end{align*}
    This concludes the proof of the lemma.
\end{proof}
\begin{remark}
 Actually, the form of the polynomials $p_m$ is well-known since it is given by the Hermite polynomials $H_m$ via
 $$ p_m(\eta) = c_m H_m(\im \eta)$$
 up to differing conventions and a multiplicative factor $c_m\in\mathbb C \setminus \{ 0 \} $. However, we only use the facts of Lemma \ref{lemma:SecondBranchSolutionsHO} in the following argument.
\end{remark}

\noindent 
Using the preceding qualitative representation of solutions, we are in the position to determine the asymptotic behaviour of $\Ups_{\mu, 0, 1}$ and $\Ups_{\mu, 0, 2}$ for $\mu\in -2\NN-1$. We consider the two solutions branches $\psi_{\mu, i}$ and $\psi_{\mu, j}$ of \Cref{lemma:SecondBranchSolutionsHO}, separately. We begin with $\psi_{\mu, i}$. 
\begin{lemma} \label{lemma:AsymptoticsOddMu1}
    Let $b = e^{\pm  \im \frac{ \pi}{8}}$ and consider   $\mu = -(2m+3)\in -(2\NN+1)$ with $m\in \mathbb N_0$. Setting the index $i = (3 +(-1)^{m+1})/2 \in \{1, 2\}$, then it holds the asymptotic expansion 
    \begin{align*}
        \Ups_{\mu, 0, i}(b z )  \quad \sim \quad 
        2^m
        (m+3)
        \frac{ \left \lceil{ \frac{m}{2} }\right \rceil!}{ \left(2\left \lceil{ \frac{m}{2} }\right \rceil\right)!}
         (bz)^{m-1} \exp\left( b^2 \frac{z^2}{2}\right) 
    \end{align*}
    as $ z \to \pm \infty$ in $\mathbb R$.
\end{lemma}
\noindent 
We note that $(bz)^{m-1}$ is single-valued for any choice of ${\rm arg}(b)$.
\begin{proof}
    We take $\psi = \psi_{\mu,i}(\eta) = p_m(\eta) e^{\eta^2/2}$ as in \Cref{lemma:SecondBranchSolutionsHO} and we determine the asymptotic of $\Ups =\Ups_{\mu, 0, i}$ through the identity
     \begin{align*}
        \Ups(\eta)= \left( -\frac{\dd^2}{\dd \eta^2} + \eta^2-\mu \right) \int_{0}^\eta \psi(\xi) \dd \xi =-\psi'(\eta)+\psi'(0) + (\eta^2 - \mu ) \int_{0}^\eta \psi(\xi) \dd \xi.
     \end{align*}
    $\psi'(0)$ is constant, while $-\psi'(\eta) = -\psi_{\mu, i}'(\eta)$  with $\eta = b z$ satisfies the identity
    \begin{align*}
        -\psi'(bz)& = -\left[p'_m(b z)+ b z\, p_m(b z)   \right] \exp\left( b^2 \frac{z^2}{2}\right).
    \end{align*}
    We have by Lemma \ref{lemma:helpfulAsymptotics} and $\eta = b z$
    \begin{align*}
      (\eta^2 - \mu ) \int_{0}^\eta \psi(\xi) \dd \xi
      =
      (b z^2 - \mu ) \int_{0}^{b z}   p_m(\xi) e^{\frac{\xi^2}{2}} \dd \xi
      \quad 
      \sim \quad  \frac{b^2 z^2-\mu}{b z} p_m(b z )  \exp\left( b^2 \frac{z^2}{2}\right).
    \end{align*}
    In sum, we have the following asymptotics of $\Upsilon = \Upsilon_{\mu, 0, i}$:
    \begin{align*}
        \Ups(b z) \quad \sim \quad -\left(   p_m'(b z)  +\mu \frac{p_m(b z )}{b z}\right) \exp\left( b^2 \frac{z^2}{2}\right),\quad \text{as }z \to \pm \infty \quad \text{in }\mathbb R.
    \end{align*}
    The right-hand side is not identically null, since the equation $\eta f'(\eta)=-\mu f(\eta) = (2m+3)f(\eta)$ is only solved by monomials  $f(\eta) = c \eta^{2m+3}$ but the polynomial $p_m$ has degree $m$. Thanks to the the leading order of $p_m$ we finally obtain that
    \begin{align*}
        \Ups_{\mu,0,i} (b z) \quad \sim \quad 
        -
        \frac{2^m \left \lceil{ \frac{m}{2} }\right \rceil!}{ \left(2\left \lceil{ \frac{m}{2} }\right \rceil\right)!}
        (m+\mu) (b z )^{m-1} \exp\left( b^2 \frac{z^2}{2}\right).
    \end{align*}
    The assertion follows with $m+ \mu = -(m+3)$.
\end{proof}
\noindent 
Next, we address the asymptotics of the second solution branch given by \Cref{lemma:SecondBranchSolutionsHO}. Since the computations are slightly more extended, we give the result in the following:
\begin{lemma} \label{lemma:AsymptoticsOddMu2}
Let $b = e^{\pm  \im \frac{ \pi}{8}}$ and consider  $\mu =  -(2m+3)\in -(2\NN+1)$ with $m\in \mathbb N_0$. Then, with index $j = (3 +(-1)^{m})/2 \in \{1, 2\}$, it holds the asymptotic expansion 
    \begin{align*}
        \Ups_{\mu, 0, j}(b z) \quad \sim \quad \pm  
        \frac{1}{\left \lceil{ \frac{m-1}{2} }\right \rceil!}
      \frac{1}{4^{\left \lceil{ \frac{m-1}{2} }\right \rceil}}
        \frac{\sqrt{\pi}}{2} (m+3) 2^m (b z)^{m-1} \exp\left( b^2 \frac{z^2}{2}\right). 
    \end{align*}
    as $ z \to \pm \infty$.
\end{lemma}
\begin{proof}
    By Lemma \ref{lemma:SecondBranchSolutionsHO}, we know that $\psi = \psi_{\mu, j}$ satisfies
    \begin{align*}
        \psi(\eta) =  q_m(\eta)  e^{\eta^2/2} \int_{0}^\eta e^{-{\xi}^2} \dd \xi + \tilde q(\eta) e^{-\eta^2/2}.
    \end{align*}  Similar to the above argument, we calculate the summands of $\Ups(\eta) = - \psi'(\eta)+\psi'(0) + (\eta^2-\mu) \int_{0}^\eta \psi(\xi)\dd \xi$. But $\psi'(0)$ is constant and for the derivative in $\eta = b z$, we obtain
    \begin{align*}
        -\psi'(b z) = -\left[q_m'(b z) +b z q_m(bz) \right] \exp\left( b^2 \frac{z^2}{2}\right)  b \int_0^{b z}\exp\left( -\xi^2\right)  \dd \xi + \hat q(b z) \exp\left( -b^2 \frac{z^2}{2}\right) ,
    \end{align*}
    where $\hat q$ is a polynomial. Note that the last summand on the right-hand side is exponentially decreasing as $z \to \pm \infty$, thus it has no relevant contribution on the asymptotics of $\Upsilon$. 

    \noindent 
    Next, we address the leading term for the asymptotics of $\int_0^\eta \psi(\xi)d\xi$ which is due to the contribution of $q_m$. We apply an integration by parts:
    \begin{align*}
        &\int_{0}^{z} q_m(b\tilde z ) \exp\left( b^2 \frac{\tilde z^2}{2}\right)  \left( \int_0^{b\tilde z} e^{-\xi^2} \dd \xi \right)  b\cdot \dd \tilde z   
        =
        \int_{0}^{z} 
        \frac{d}{d \tilde z}
        \left(
           \int_0^{\tilde z} q_m(b s) \exp\left( b^2 \frac{s^2}{2}\right)ds
        \right)
        \left( \int_0^{b\tilde z} e^{-\xi^2} \dd \xi \right)  b\cdot \dd \tilde z
        \\
        &=
        \left\{
            \left(
            \int_{0}^{\tilde z} q_m(bs ) \exp\left( b^2 \frac{s^2}{2}\right) \dd s 
            \right)
            \left(
            \int_0^{b\tilde z} e^{-\xi^2} \dd \xi 
            \right)
        \right\}
        \bigg|_{\tilde z = 0}^{\tilde z = z}  
        - \int_{0}^z \left( \int_{0}^{\tilde z} q_m(bs ) \exp\left( b^2 \frac{s^2}{2}\right) \dd s   \right) \cdot  b  \exp\left( -b^2 \tilde z^2 \right) \dd \tilde z \\
        & \sim \quad \int_{0}^z q_m(bs ) \exp\left( b^2 \frac{s^2}{2}\right)  \dd s \cdot b \int_0^z \exp\left(- b^2 \tilde z^2\right) \dd \tilde z   
        \sim \quad \frac{q_m(b z)}{b z} \exp\left( b^2 \frac{z^2}{2}\right) \int_0^z \exp\left(- b^2 \tilde z^2\right)  \dd \tilde z. 
    \end{align*}
    Note that, passing the summand with negative sign in the second line is of lower order (asymptotically constant), while in the last relation we have used Lemma \ref{lemma:helpfulAsymptotics}.
    In sum, it holds the following asymptotics
    \begin{align*}
        \Ups_{\mu, 0, j} (b z) & = -\psi'(b z) +  b(b^2z^2-\mu) \int_{0}^z \psi(b\tilde z) \dd \tilde z  \\ 
        &\sim \quad 
        - 
        \left(q_m'(b z) + \mu\frac{q_m(bz)}{b z} \right) \exp\left( b^2 \frac{z^2}{2}\right)b\int_0^z \exp\left( -b^2 \tilde z^2\right)  \dd \tilde z  \\
        & \sim \quad \mp 
         \frac{1}{\left \lceil{ \frac{m-1}{2} }\right \rceil!}
      \frac{1}{4^{\left \lceil{ \frac{m-1}{2} }\right \rceil}}
        \frac{\pi}{2} (m+\mu) 2^m (b z)^{m-1} \exp\left( b^2 \frac{z^2}{2}\right). 
    \end{align*}
    The claim follows from $m+\mu =-(m+3)$.
\end{proof}

\subsection{Exact formulation for $\mu =-1$}$\,$

\smallskip
\noindent 
As a consequence we can extend the assertion of Lemma \ref{lemma:NonOddMuAsymptotics}  to any $\mu$ lying on the negative real axis with the exception of the crucial value $\mu=-1$. However, we are able to give explicit solutions in this case:
\begin{lemma}  \label{lemma:AsymptoticsMu=-1}
Let $\eta \in \mathbb C$ and consider  $\mu = -1$. Then it holds   \begin{align*}
    \Ups_{-1,0,1}(\eta)& =  
    \frac{1}{2}
    \left(
        \eta e^{-\frac{\eta^2}{2}}
        +
        (1+\eta^2)
        \sqrt{\frac{\pi}{2}}
        \erf
        \left(
            \frac{\eta}{\sqrt{2}}
        \right)
    \right) \\
    \Ups_{-1,0,2}(\eta) &=\sqrt{\frac{\pi}{ 2} }
    \left(
        1+ \frac{\eta^2}{2}
    \right)
    \int_0^\eta 
    \erf\im (\xi)
    e^{ -\frac{\xi^2}{2}}
    d\xi 
    -
    \frac{\sqrt{\pi}}{\, 4} 
    \int_0^\eta 
    \xi^2
    \erf\im (\xi)
    e^{ -\frac{\xi^2}{2}}
    d\xi 
    -
    \frac{1}{2}.
\end{align*}
\end{lemma}

\subsection{Proof of Corollary \ref{cor:GrowthOnR}}$\,$

\smallskip
\noindent 
Finally, we revise the obtained asymptotics to address the statement of Corollary \ref{cor:GrowthOnR}. 
By Proposition \ref{prop:RepresentationByPsi}, we know the general solution of equation \ref{eq:intro-Upsilon-equation} is given by 
\begin{align*}
    \Ups (\eta) &= 
    c_0(\mu-\eta^2)(\eta) +c_1\Ups_{\mu, 0, 1}(\eta) +c_2\Ups_{\mu, 0, 2}(\eta).
\end{align*}
Clearly, the first term $\mu-\eta^2$ grows quadratically. In view  of the previous lemmata, setting $\eta =b z$ with $b := e^{\pm \frac{\pi}{8}\im}$ we hence need to exclude the possibility that the different summands cancel each other at both $z \to \pm \infty$ for certain combinations of coefficients $c_0,c_1,c_2$. We proceed by investigating three cases: \\
\textit{Case $\mu=-1$:}  Referring to Lemma \ref{lemma:AsymptoticsMu=-1},  the function $\Upsilon_{-1, 0, 1}(b z)$ grows quadratically on both side of $\R$ with odd parity. Regarding $\Ups_{-1, 0,2}(b z)$, the asymptotic behaviour is exponential by the same lemmata.  Hence for every combination of $c_0$ and $c_1$ we must have 
\begin{align*}
     \lim_{\substack{\eta \to \infty, \\ |\arg(\eta) |< \frac{\pi}{4} }} 
     \left|
        \frac{\Upsilon(\eta)}{\mu-\eta^2}
    \right|
    \neq 0
    \qquad 
    \text{or}
    \qquad 
     \lim_{\substack{\eta \to \infty, \\ |\arg(-\eta) |< \frac{\pi}{4} }} 
     \left|
        \frac{\Upsilon(\eta)}{\mu-\eta^2}
    \right|
     \neq 0.
\end{align*}
\textit{Case} $\mu \in -(2\NN+1)$: \\
The asymptotic behaviour of the summands of $\Ups_{\mu, 0,1}$ and $\Ups_{\mu, 0,1}$ are given in Lemma \ref{lemma:AsymptoticsOddMu1} and Lemma \ref{lemma:AsymptoticsOddMu2}. In particular, these asymptotics are of exponential order but with differing parities for   $\Ups_{\mu, 0,1}$ and $\Ups_{\mu, 0,2}$. Hence for every combination  $(c_1,c_2)\neq (0,0)$, we have  
\begin{equation*}
     \lim_{\substack{\eta \to \infty, \\ |\arg(\eta) |< \frac{\pi}{4} }} 
     \left|
        \frac{\Upsilon(\eta)}{\mu-\eta^2}
    \right|
    =
    +\infty 
    \qquad 
    \text{or}
    \qquad 
     \lim_{\substack{\eta \to \infty, \\ |\arg(-\eta) |< \frac{\pi}{4} }} 
     \left|
        \frac{\Upsilon(\eta)}{\mu-\eta^2}
    \right|
     =
     + \infty.
\end{equation*}
\textit{Case} $ \mu \in \CC \setminus (2\ZZ-1)$: \\
For all remaining $\mu \in \CC \setminus (2\ZZ-1)$, we realize that the asymptotics of \Cref{cor:final-asymptotic-Ups-mu-not-integer} implies that $\Ups_{\mu, 0, 1}$ and $\Ups_{\mu, 0, 2}$ grow exponentially with different parities. Hence any linear combination needs to grow at least for $z \to +\infty$ or $z \to - \infty$ exponentially. This proves the claim. 

\section{Some examples of explicit solutions}\label{sec:application-of-the-method}
\noindent 
In this section, we present two examples illustrating how to generate solutions to Equation \eqref{eq:linearised-Prandtl}, following the methodology introduced in \Cref{rmk:method-to-build-solutions}. To simplify the calculations, we focus on the specific case where $U_{\rm sh}(y) = -y^2$, corresponding to $\alpha = 0$, $\beta = -1$, and $a = 0$.  

\noindent 
We consider positive frequencies $k > 0$ and, in both examples, set $\sigma = e^{\frac{7\pi}{4} \im} = (1 - \im)/\sqrt{2} \in \mathbb{C}$, which implies $\tau = -\im \sigma = e^{\frac{5\pi}{4} \im} = -(1+\im)/\sqrt{2} \in \mathbb{C}$ and $\mu = - \sigma e^{\frac{\pi  }{4}\im} = -1$. Next, we compute the two solutions $\psi_{\mu, 1}$ and $\psi_{\mu, 2}$, which we abbreviate as $\psi_1$ and $\psi_2$:
\begin{equation*}
\begin{alignedat}{4}
    \psi_{ 1}(\eta) 
    &= 
    \; 
    \mathcal M \Big( 0 , \frac{1}{2}, \eta^2 \Big)
        e^{ -\frac{\eta^2}{2}}
    &&=
    e^{ -\frac{\eta^2}{2}},\\
    \psi_{2}(\eta) 
    &= 
    \eta 
    \mathcal M \Big( \frac{1}{2} , \frac{3}{2}, \eta^2 \Big)
    e^{ -\frac{\eta^2}{2}}
    &&= \frac{\sqrt{\pi}}{\, 2} \erf\im (\eta)
    e^{ -\frac{\eta^2}{2}},
\end{alignedat}        
\end{equation*}
where $\erf \im(\eta) = \frac{2}{\sqrt{\pi}} \int_0^\eta e^{\zeta^2} d\zeta$. We remark that $\eta_* = -a \sqrt[4]{|\beta||k|}$ is null for any $k \in \mathbb{N}$, thus we can compute the functions $\Upsilon_{\mu,\eta_*, 1}$ and $\Upsilon_{\mu, \eta_*,2}$ (abbreviated as $\Upsilon_{1}$ and $\Upsilon_{2}$) through the expressions
\begin{equation}\label{eq:Upsilon1-associated-to-W-criterium}
    \Upsilon_{1}(\eta) = 
    \int_0^\eta 
    \left(
        1+ \frac{\eta^2-\xi^2}{2}
    \right)
    \psi_1(\xi) d\xi 
    -
    \frac{\psi_1'(0)}{2}
    =
    \frac{1}{2}
    \left(
        \eta e^{-\frac{\eta^2}{2}}
        +
       (1+\eta^2)
        \sqrt{\frac{\pi}{2}}
        \erf
        \left(
            \frac{\eta}{\sqrt{2}}
        \right)
    \right)
\end{equation}
while for $\Upsilon_2(\eta)$, we obtain
\begin{align*}
    \Upsilon_{2}(\eta) 
    &= 
    \int_0^\eta 
    \left(
        1+ \frac{\eta^2-\xi^2}{2}
    \right)
    \psi_2(\xi) d\xi 
    -
    \frac{\psi_2'(0)}{2}\\
    &=
    \frac{\sqrt{\pi}}{\, 2} 
    \left(
        1+ \frac{\eta^2}{2}
    \right)
    \int_0^\eta 
    \erf\im (\xi)
    e^{ -\frac{\xi^2}{2}}
    d\xi 
    -
    \frac{\sqrt{\pi}}{\, 4} 
    \int_0^\eta 
    \xi^2
    \erf\im (\xi)
    e^{ -\frac{\xi^2}{2}}
    d\xi 
    -
    \frac{1}{2}.
\end{align*}
It is possible to further develop the last integral in terms of Owen's T function, but this lies beyond the scope of the present paper. Nevertheless, we note that $\Upsilon_{2}$ is well defined, entire, and can be computationally evaluated. 

\noindent
The general solution of the Prandtl eigenproblem \eqref{eq:intro-Upsilon-equation} is therefore
\begin{equation*}
    \Upsilon(\eta) = c_0 (-1-\eta^2) + c_1 \Upsilon_1(\eta)+c_2 \Upsilon_2(\eta), \qquad \text{with }c_0, c_1,c_2\in \mathbb C.
\end{equation*}
If we do not enforce the no-slip boundary conditions at $y = 0$, the constants can be chosen arbitrarily. For instance, setting $c_0 = -\sqrt{\frac{\pi}{2}}$, $c_1 = 2$, and $c_2 = 0$, the function $\Upsilon(\eta)$ simplifies to 
\begin{equation}\label{eq:an-interesting-ups}
    \Upsilon(\eta) 
    =
        \eta e^{-\frac{\eta^2}{2}}
        +
        \sqrt{\frac{\pi}{2}}
        (1+\eta^2)
        \left(
        \erf
        \left(
            \frac{\eta}{\sqrt{2}}
        \right)
        -1
        \right).
\end{equation}
The corresponding stream function is given by $\phi_k(y) = \Upsilon(e^{-\frac{\pi}{8 }\im}\sqrt[4]{k} \, y)$, leading to the velocity components
$u(t,x,y) = \phi_k'(y) e^{\im kx + t \sqrt{k}\frac{1-\im}{\sqrt{2}} }$ and $v(t,x,y) = -\im k \phi_k(y) e^{\im kx + t \sqrt{k}\frac{1-\im}{\sqrt{2}} }$.
These satisfy Equations \eqref{eq:linearised-Prandtl}. However, we note that $\phi_k(0) = -\sqrt{\frac{\pi}{2}} \neq 0$, as well as $\phi_k'(0)\neq 0$. The plot of $\phi_k$ is provided in \Cref{fig:Plotphik}.
\begin{figure}[t]
    \centering
    \includegraphics[width=7cm]{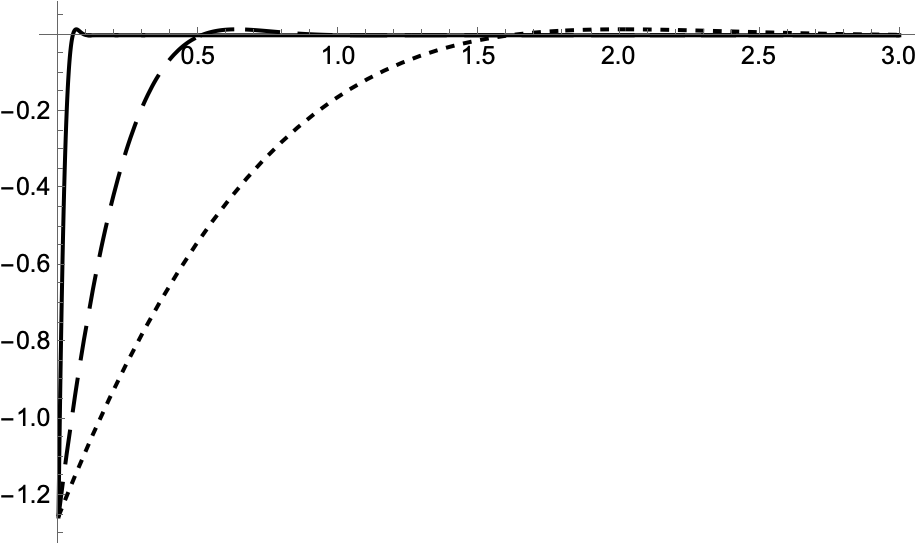}
    \hspace{1cm} 
    \includegraphics[width=7cm]{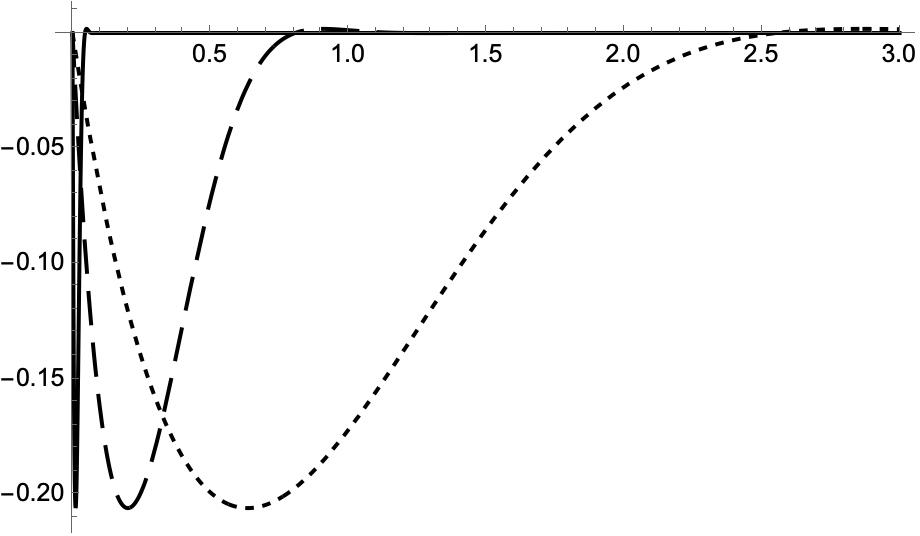}
    \caption{
    Plot of the stream function $\phi_k= \phi_k(y)$ defined by \eqref{eq:an-interesting-ups} for $k = 1$ (dotted line), $k =10^2$ (dashed line) and $k = 10^6$ (full line). With $\sigma = (1-\im)/\sqrt{2}$, thus this profile inflates in time as $\exp(t\sqrt{k}/{2\sqrt{2}})$. The no-slip boundary conditions $\phi_k(0) = \phi_k'(0) = 0$ are not satisfied. Left and right correspond to the real and imaginary parts, respectively.
    }
    \label{fig:Plotphik}
\end{figure}

\noindent
To pursue the no-slip boundary conditions $\phi_k(0) = \phi_k'(0) = 0$, we need to impose \eqref{eq:main-thm-constants-for-bdycdt} and we first compute $\psi_1(0)= 1$, $\psi_2(0)=0 $, $\psi_1'(0) =0$ and $ \psi_2'(0) =1 $. The constants must satisfy $2c_0+c_2 = 0$ and $c_1 = 0$. Choosing $c_0 =- 1/2$ and $c_2 = 1$ leads to
\begin{equation}\label{eq:an-interesting-ups2}
    \Upsilon(\eta) 
    =
    \frac{\eta^2}{2}
    +
    \frac{\sqrt{\pi}}{\, 2} 
    \left(
        1+ \frac{\eta^2}{2}
    \right)
    \int_0^\eta 
    \erf\im (\xi)
    e^{ -\frac{\xi^2}{2}}
    d\xi 
    -
    \frac{\sqrt{\pi}}{\, 4} 
    \int_0^\eta 
    \xi^2
    \erf\im (\xi)
    e^{ -\frac{\xi^2}{2}}
    d\xi.
\end{equation}
The corresponding stream function $\phi_k(y) =\Upsilon(e^{-\frac{\pi }{8} \im}\sqrt[4]{k} \, y)$ is plotted in \Cref{fig:Plotphik2}. Unfortunately, $\phi_k(y)$ grows and oscillate in space as $\phi_k(y) \approx e^{\sqrt{k} \frac{1-\im}{\sqrt{2}}\frac{y^2}{2}}$ making it not amenable for the ill-posedness of \eqref{eq:linearised-Prandtl} in Gevrey-classes.
\begin{figure}[t]
    \centering
    \includegraphics[width=7cm]{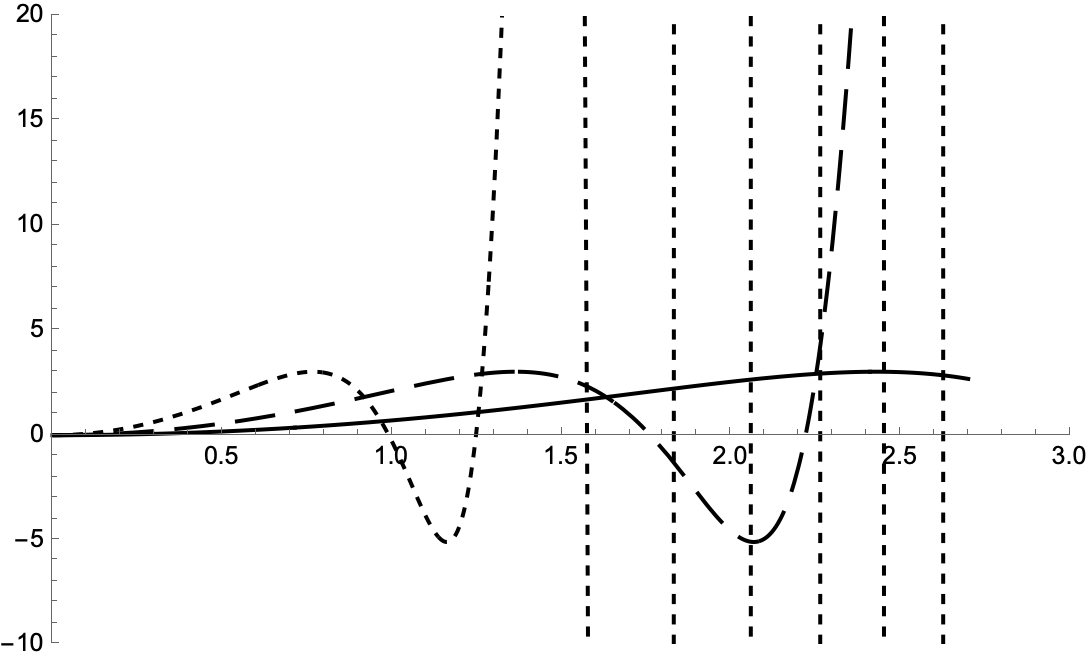}
    \hspace{1cm} 
    \includegraphics[width=7cm]{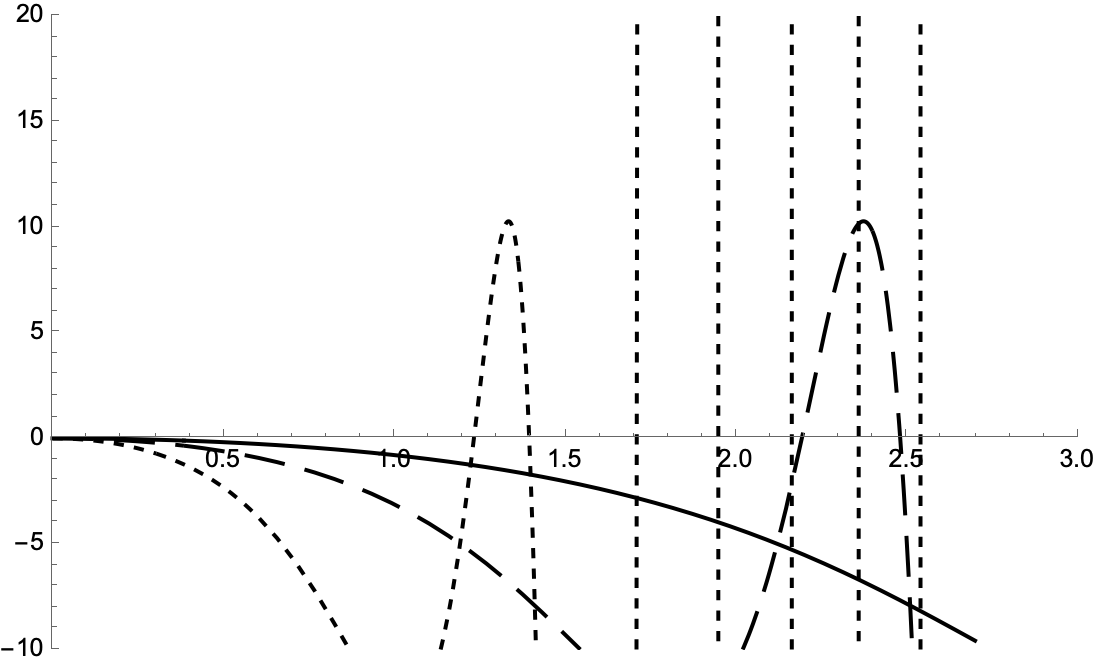}
    \caption{
    Plot of the stream function $\phi_k= \phi_k(y)$ defined by \eqref{eq:an-interesting-ups2} for $k = 1$ (full line), $k =10$ (dashed line) and $k = 10^2$ (dotted line). With $\sigma = (1-\im)/\sqrt{2}$,  this profile inflates in time as $\exp(t\sqrt{k}/{\sqrt{2}})$. The no-slip boundary conditions $\phi_k(0) = \phi_k'(0) = 0$ are satisfied but generates exponential growth and oscillations in $y$. Left and right correspond to the real and imaginary parts, respectively.
    }
    \label{fig:Plotphik2}
\end{figure}

\section{An explicit formulation of the shear-layer velocity}\label{sec:explicit-form-of-shear-layer}

\noindent 
In this final section, we highlight how our \Cref{cor:main-result-tau-W-unique} enables us to obtain, with analytical precision, the shear-layer velocity of G\'erard-Varet and Dormy \cite{MR2601044}, which has been asymptotically associated with the instability of the Prandtl equations. The following formula can be applied to any shear flow $U_{\rm sh}(y)$ that is sufficiently regular and satisfies $U_{\rm sh}'(a) = 0$ and $U_{\rm sh}''(a) < 0$ (not only the quadratic case of \eqref{parabolic-shear-flow}).

\noindent
For the sake of comparison, we adopt the same notation as in \cite{MR2601044}. Setting $\ee = 1/k$, the shear-layer velocity $v^{\rm sl}_\ee$ is defined as
\begin{equation}\label{eq:shear-layer}
    v^{\rm sl}_\ee (y) = \ee^{\frac{1}{2}}V\left(\frac{y-a}{\ee^{\frac{1}{4}}}\right)
    = 
    \ee^{\frac{1}{2}}V\left(z\right),\qquad 
    \text{with }z = \frac{y-a}{\ee^{\frac{1}{4}}}.
\end{equation}
(cf.~formula below (2.5)  in \cite{MR2601044}). Making use of the definition of $V$ and $\tilde V$ at the pages 596 and 597 of  \cite{MR2601044} (being careful that $\tau$ here is $\tilde \tau$ in \cite{MR2601044}), one has that
\begin{equation*}
\begin{aligned}
       \tilde V(z) := 
       V(z) 
       +
       \left(
            \frac{|U_{\rm sh}''(a)|^\frac{1}{2}}{\sqrt{2}}
            \tau + U_{\rm sh}''(a)\frac{z^2}{2}
       \right)
       H(z) = 
       \left(
            \frac{|U_{\rm sh}''(a)|^\frac{1}{2}}{\sqrt{2}}
            \tau + U_{\rm sh}''(a)\frac{z^2}{2}
       \right)
       W
       \left(
            \sqrt[4]{
            \frac{|U_{\rm sh}''(a)|}{2}
            }
            z
        \right),
\end{aligned}
\end{equation*}
with $H$ denoting the Heaviside function. On the other hand, \Cref{cor:main-result-tau-W-unique} implies that $\tau = e^{\frac{5\pi}{4} \im}$ and that $W$ can be expressed in terms of $\erf$ and other elementary functions. More precisely, our result implies that the function $V$ is given by
\begin{equation}\label{eq:V-Gerard-Dormy-explicit}
\begin{aligned}
    V(z) &= 
    \left(
            \frac{|U_{\rm sh}''(a)|^\frac{1}{2}}{\sqrt{2}}
            \tau + U_{\rm sh}''(a)\frac{z^2}{2}
       \right)
    \left(
         W
       \left(
            \sqrt[4]{
            \frac{|U_{\rm sh}''(a)|}{2}
            }
            z
        \right)
        -H(z)
    \right)\\
    &=
    e^{\frac{5 \pi}{4}\im }
    \frac{|U_{\rm sh}''(a)|^\frac{1}{2}}
    {\sqrt{2}}
    \left(
            1
            +
            f(z)^2
    \right)
     \left(
        \frac 12
        +
        \frac 12
        \erf
        \left( 
            \frac{f(z)}{\sqrt{2}}
        \right)
     +
     \frac{1}{\sqrt{2\pi}}
        \frac{f(z)}{ 1+f(z)^2}
        e^{
            - \frac{f(z)^2}{2}
        }
        -H(z)
        \right),
\end{aligned}
\end{equation}
where the function $f(z)$ is defined as
\begin{equation*}
    f(z) :=
     \frac{|U_{\rm sh}''(a)|^\frac{1}{4}}{\sqrt[4]{2}}
    \eta(z)
    =
    \frac{|U_{\rm sh}''(a)|^\frac{1}{4}}{\sqrt[4]{2}e^{\frac{ \pi}{8}\im}}z.
\end{equation*}
Expression \eqref{eq:V-Gerard-Dormy-explicit} generates the ``shear-layer'' velocity \eqref{eq:shear-layer} for any $U_{\rm sh}(y)$ satisfying $U_{\rm sh}'(a) = 0$ and $U_{\rm sh}''(a) < 0$, depending only on $U_{\rm sh}''(a)$.

\noindent 
We can plot the formula \eqref{eq:V-Gerard-Dormy-explicit} using the example from Section 5.2 of \cite{MR2601044}, where $U_{\rm sh}(y) = 2y \exp(-y^2)$. The singular point is $a = 1/\sqrt{2} > 0$, with $U_s'(1/\sqrt{2}) = 0$ and $U_{\rm sh}''(1/\sqrt{2}) = -4\sqrt{2}/\sqrt{e}$. The plot of the resulting function $V$ is shown in \Cref{fig:Plotphik3}, which precisely matches the plot of the shear-layer correction $v_{\rm in}^{\rm th}$ in \cite{MR2601044} (cf.~Figure 3, page~607).
\begin{figure}[t]
    \centering
    \includegraphics[width=8.5cm]{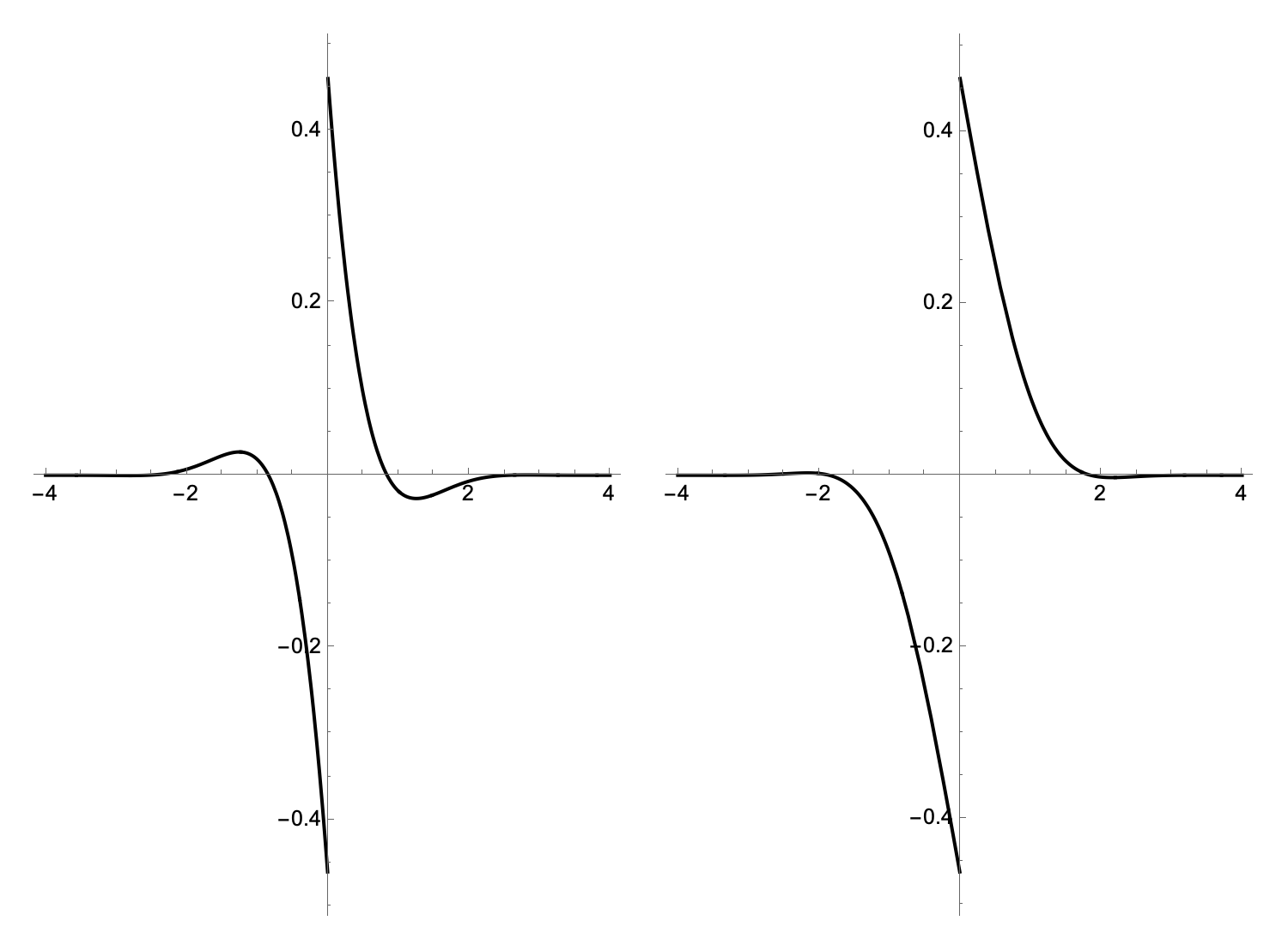}
    \caption{
    Plot of the ``shear-layer'' function $V$ \eqref{eq:V-Gerard-Dormy-explicit}, explicitly defined in terms of the $\erf$ and $\exp$ functions in $z$. It coincides with the shear-layer correction in \cite{MR2601044} for the shear layer $U_{\rm sh}(y) = 2y \exp(-y^2)$. 
    Left and right correspond to the real and imaginary parts, respectively.
    }
    \label{fig:Plotphik3}
\end{figure}

\appendix
\section{The  harmonic oscillator}\label{appendix-A}

\noindent 
For the sake of completeness, we give some facts on the solutions of the harmonic oscillator. First, it follows a representation of the algebraic solutions of the harmonic oscillator in terms of confluent hypergeometric functions:
\begin{lemma} \label{lemma:HarmonicOscillatorSolutions}
    Every solution $\psi:\CC  \to \CC$ of the differential equation
    \begin{align} \label{eq:harmonicOscillator3}
        -\psi'' + \eta^2 \psi= (\mu +2) \psi, \qquad \mu \in \CC 
    \end{align}
    is of the form 
    \begin{align*}
        \psi(\eta) = e^{-\frac{\eta^2}{2}} \left[ \tilde c_1  \MB \left( - \frac{1+\mu}{4} , \frac12, \eta^2 \right) + \tilde c_2 \eta \MB \left(  \frac{1-\mu}{4} , \frac32, \eta^2 \right) \right]
    \end{align*}
    for arbitrary $\tilde c_1 , \tilde c_2 \in \CC$.
\end{lemma}
\begin{proof}
    Following well-known reductions from quantum physics (e.g.\ \cite{straumann2012quantenmechanik}), we have that 
    $ \phi(\eta) := e^{ \frac{\eta^2}{2}} \psi(\eta)$ solves
    \begin{align*}
        \phi'' -2 \eta \phi' +(\mu+1)\phi=0.
    \end{align*}
    Setting $w(z) := \MB(a,c,z)$, we have that $z w''(z)+ (b-z)w'(z)-aw(z)=0$ for any $z \in \CC$. For $a= - \frac{1+\mu}{4}, c=\frac12$ and $z =\eta^2 $, we have 
    \begin{align*}
       \frac{\dd^2}{\dd^2 \eta} & w(\eta^2) -2\eta \frac{\dd }{\dd \eta} w(\eta^2) +(\mu+1) w(\eta^2) \\
       & = 4 \left[ \eta^2 w''(\eta^2)  + \frac12 w'(\eta) - \eta^2 w'(\eta^2) - \left( - \frac{1+\mu}{4}\right) w(\eta^2) \right] \\
       & =0.
    \end{align*}
    Analogously, for $ a= \frac{1-\mu}{4}, c=\frac32$ and $\bar w(\eta) := \eta w(\eta^2) $, it is
    \begin{align*}
       \frac{\dd^2}{\dd^2 \eta} & \bar w(\eta) -2\eta \frac{\dd }{\dd \eta}\bar w(\eta) +(\mu+1) \bar  w(\eta) \\
       & = 4 \eta^3 w''(\eta^2) +6 \eta w'(\eta^2) -4 \eta^3 w'(\eta^2) -2 \eta w(\eta^2 ) +(\mu+1) \eta w(\eta^2) \\
       &= 4\eta \left[ \eta^2 w''(\eta^2) +\left(\frac32-\eta^2\right) w'(\eta^2) - \frac{1-\mu}4 w(\eta^2)\right] \\
       &=0.
    \end{align*}
    As both functions, $\eta \mapsto w(\eta^2), \bar w (\eta)$, are linearly independent, the assertion follows.
\end{proof}
\noindent Another important tool consists of the up- and down-operators
\begin{align*}
    \Acal_{\uparrow}: = \eta - \frac{d}{d\eta},\qquad \Acal_{\downarrow}:= \eta+\frac{d}{d \eta}.
\end{align*}
In our notation, a given solution $ \psi_{\mu}$ \eqref{eq:harmonicOscillator3}  transcends to another solution of the algebraic eigenvalue $\mu\pm2$ by applying the up- respectively down-operator. It holds:
\begin{lemma}   \label{lemma:UpDownOperator}
    If $\psi = \psi_\mu:\CC \to \CC$ is a solution to \eqref{eq:harmonicOscillator3}, then
    \begin{align}  \label{eq:UpAndDownOperator}
        \psi_{\mu+2} := \Acal_{\uparrow}\psi_\mu , \qquad  \psi_{\mu-2} :=\Acal_{\downarrow}\psi_{\mu} 
    \end{align}
    are solutions of \eqref{eq:harmonicOscillator3} with $\mu$ replaced by $\mu+2$ and $\mu-2$, respectively.
\end{lemma}
\begin{proof}
We prove the first equality of \eqref{eq:UpAndDownOperator} as the second one follows analogously. Recall that $\Bcal_\mu=-\frac{d^2}{d\eta^2}+ \eta^2-\mu$. Then we compute the commutator
\begin{align*}
    [\Bcal_{\mu},\Acal_\uparrow] &= \left[-\frac{d^2}{d\eta^2}, \eta\right] - \left[ \eta^2, \frac{d}{d \eta}\right]   =  -2 \frac{d}{d \eta } +2 \eta   = 2 \Acal_\uparrow,
\end{align*}
for any $\mu \in \CC$. With this at hand, we have
    \begin{align*}
        \Bcal_{\mu+2} \Acal_{\uparrow} \psi_\mu & = \Acal_{\uparrow}\Bcal_{\mu+2}\psi_{\mu}+[ \Bcal_{\mu+2},\Acal_{\uparrow}]\psi_\mu  =   2\Acal_\uparrow \psi_\mu.
    \end{align*}
This proves the claim.
\end{proof}
\noindent It should be pointed out that the operator $\Acal_\uparrow$ and $\Acal_\downarrow$ have one-dimensional kernels coinciding with solutions to \eqref{eq:harmonicOscillator3}, namely 
\begin{align*}
    \ker \Acal_\uparrow  &= \langle \eta \mapsto e^{-\eta^2/2} \rangle = \langle \psi_{-1,1} \rangle, \\
    \ker \Acal_\downarrow  &= \langle \eta \mapsto e^{\eta^2/2} \rangle = \langle \psi_{-3,1} \rangle.
\end{align*}
Hence, when constructing new solutions by those operators, one needs to ensure that these are non-vanishing. With this in mind, we give an auxiliary result concerning solutions to negative odd eigenvalues $\mu$.
\begin{lemma}\label{lemma:appx-solutions-harmonic-osc-}
    Let $\mu \in -(2\NN+1)$ and $m=- \frac{\mu+3}{2} \in \NN_0$. Then,  with indexes $i = (3 +(-1)^{m+1})/2 \in \{1, 2\}$ and $j = (3 +(-1)^{m})/2 \in \{1, 2\}$, for every $m \in \NN $, the two linearly independent solutions to \eqref{eq:harmonicOscillator3} are given by 
    \begin{align}  \label{eq:harmonicOscillatorSolutionsNegativeOdd}
      \psi_{\mu,i}(\eta) &= \Acal_\downarrow^m e^{\eta^2/2} =e^{-\eta^2/2} \frac{\dd^{m} }{\dd \eta^{m}} e^{\eta^2}
      ,
      \qquad 
      \qquad 
      \psi_{\mu,j}(\eta) = 
      \Acal_\downarrow^m \left( e^{\eta^2/2} \int_{0}^\eta e^{-{\xi}^2} \dd \xi\right)
    \end{align}
    up to a non-vanishing multiplicative factor.
\end{lemma}
\begin{proof}
   In case $m=0$ which is $\mu=-3$, a simple computation shows 
    \begin{align*}
       \psi_{-3,1}= e^{\eta^2/2}, \qquad \qquad   \psi_{-3,2} = e^{\eta^2/2} \int_0^\eta e^{-\xi^2} \dd \xi.
    \end{align*}
    Furthermore, it holds $\Acal_\downarrow (e^{-\eta^2/2} f(\eta) ) = e^{-\eta^2/2} f'(\eta)$, which proves the second equality in \eqref{eq:harmonicOscillatorSolutionsNegativeOdd}. The remaining claims follow by induction over $m\in \NN_0$ using Lemma \ref{lemma:UpDownOperator} and realizing that the functions $\psi_{\mu,i}$ and $\psi_{\mu,j}$ do not lie in $\ker \Acal_\downarrow$.
\end{proof}
\bibliographystyle{abbrv}
\bibliography{literature}

\end{document}